\documentclass{amsart}
\usepackage{graphics,psfrag}
\usepackage{pgf,tikz}
\usetikzlibrary{arrows}
\usepackage{hyperref}
\DeclareMathOperator{\val}{val}
\DeclareMathOperator{\rep}{rep}

\theoremstyle{plain}
\newtheorem{theorem}{Theorem}
\newtheorem{lemma}[theorem]{Lemma}
\newtheorem{corollary}[theorem]{Corollary}
\newtheorem{prop}[theorem]{Proposition}
\newtheorem{conj}[theorem]{Conjecture}
\theoremstyle{definition}
\newtheorem{definition}[theorem]{Definition}
\newtheorem{example}[theorem]{Example}
\newtheorem{remark}[theorem]{Remark}
\begin{document}
\title[Counting non-zero coefficients in Pascal triangles]{Counting the number of non-zero coefficients in rows of generalized Pascal triangles}
\author{Julien LEROY}
\thanks{J. Leroy is FNRS post-doctoral fellow.}
\author{Michel RIGO}
\author{Manon STIPULANTI}
\thanks{M. Stipulanti is supported by a FRIA grant.}
\address{Universit\'e de Li\`ege, Institut de math\'ematique, All\'ee de la d\'ecouverte 12 (B37),
4000 Li\`ege, Belgium\newline
J.Leroy@ulg.ac.be, M.Rigo@ulg.ac.be, M.Stipulanti@ulg.ac.be}

\begin{abstract}
    This paper is about counting the number of distinct (scattered) subwords occurring in a given word. More precisely, we consider the generalization of the Pascal triangle to binomial coefficients of words and the sequence $(S(n))_{n\ge 0}$ counting the number of positive entries on each row. By introducing a convenient tree structure, we provide a recurrence relation for $(S(n))_{n\ge 0}$. This leads to a connection with the $2$-regular Stern--Brocot sequence and the sequence of denominators occurring in the Farey tree. Then we extend our construction to the Zeckendorf numeration system based on the Fibonacci sequence. Again our tree structure permits us to obtain recurrence relations for and the $F$-regularity of the corresponding sequence.
\end{abstract}

\maketitle

%===============================================
\section{Introduction}
%===============================================

 A {\em finite word} is simply a finite sequence of letters belonging to a finite set called the {\em alphabet}. We recall from combinatorics on words that the binomial coefficient $\binom{u}{v}$ of two finite words $u$ and $v$ is
the number of times $v$ occurs as a subsequence of $u$ (meaning as a
``scattered'' subword). As an example, consider the following binomial coefficient for two words over $\{0,1\}$ 
$$\binom{101001}{101}=6.$$
Indeed, if we index the letters of the first word $u_1u_2\cdots u_6=101001$, we have 
$$u_1u_2u_3=u_1u_2u_6=u_1u_4u_6=u_1u_5u_6=u_3u_4u_6=u_3u_5u_6=101.$$
Observe that this concept is a natural generalization of the binomial coefficients of integers. For a one-letter alphabet $\{a\}$, we have
\begin{equation}
    \label{eq:1lettre}
\binom{a^m}{a^n}=\binom{m}{n},\quad \forall\, m,n\in\mathbb{N}
\end{equation}
where $a^m$ denotes the concatenation of $m$ copies of the letter $a$. For more on these binomial coefficients, see, for instance, \cite[Chap.~6]{Lot}. There is a vast literature on the subject with applications in formal language theory (e.g., Parikh matrices, $p$-group languages or piecewise testable languages \cite{KNS1,KNS}) and combinatorics on words (e.g., avoiding binomial repetitions \cite{Rao}). One of the first combinatorial questions was to determine when it is possible to uniquely reconstruct a word from some of its binomial coefficients; for instance, see \cite{erdos}.

The connection between the Pascal triangle and the Sierpi\'nski gasket is well understood. 
Considering the intersection of the lattice $\mathbb{N}^2$ with the region $[0,2^j]\times [0,2^j]$, the first $2^j$ rows and columns of the usual Pascal triangle $(\binom{m}{n}\bmod{2})_{m,n< 2^j}$ provide a coloring of this lattice. If we normalize this region by a homothety of ratio $1/2^j$, we get a sequence in $[0,1]\times [0,1]$ converging, for the Hausdorff distance, to the Sierpi\'nski gasket when $j$ tends to infinity.
In~\cite{LRS} we extend this result for a generalized Pascal triangle defined below.

\begin{definition}\label{def:L2}
We let $\rep_2(n)$ denote the (greedy) base-$2$ expansion of $n\in\mathbb{N}\setminus\{0\}$ starting with~$1$ (these expansions correspond to expansions with most significant digit first). We set $\rep_2(0)$ to be the empty word denoted by $\varepsilon$. Let $L_2=\{\varepsilon\}\cup 1\{0,1\}^*$ be the set of base-$2$ expansions of the integers. For all $i\ge 0$, we denote by $w_i$ the $i$th word in $L_2$. Observe that $w_i = \rep_2(i)$ for $i\ge 0$.
\end{definition}
To define an array $\mathsf{P}_2$, we will consider all the words over $\{0,1\}$ (starting with~$1$) and we order them by the radix order (i.e., first by length, then by the classical lexicographic ordering for words of the same length assuming $0<1$). For all $i,j\ge 0$, the element in the $i$th row and $j$th column of the array $\mathsf{P}_2$ is defined as the binomial coefficient $\binom{w_i}{w_j}$ of the words $w_i,w_j$ from Definition~\ref{def:L2}. This array is a generalized Pascal triangle that was introduced in \cite{LRS} and its first few values are given in Table~\ref{tab:tp}. 
\begin{table}[h!t]
$$\begin{array}{r|r|cccccccc|c}
&&\varepsilon&1&10&11&100&101&110&111&S\\
\hline
\rep_2(0)& \varepsilon&\mathbf{1} & 0 & 0 & 0 & 0 & 0 & 0 & 0 & 1\\
\rep_2(1)& 1&\mathbf{1} & \mathbf{1} & 0 & 0 & 0 & 0 & 0 & 0 &2\\
\rep_2(2)& 10&1 & 1 & 1 & 0 & 0 & 0 & 0 & 0 &3\\
\rep_2(3)& 11&\mathbf{1} & \mathbf{2} & 0 & \mathbf{1} & 0 & 0 & 0 & 0 &3\\
\rep_2(4)& 100&1 & 1 & 2 & 0 & 1 & 0 & 0 & 0 &4\\
\rep_2(5)& 101&1 & 2 & 1 & 1 & 0 & 1 & 0 & 0 &5\\
\rep_2(6)& 110&1 & 2 & 2 & 1 & 0 & 0 & 1 & 0 &5\\
\rep_2(7)& 111&\mathbf{1} & \mathbf{3} & 0 & \mathbf{3} & 0 & 0 & 0 & \mathbf{1} &4\\
\end{array}$$
\caption{The  first few values in the generalized Pascal triangle $\mathsf{P}_2$.}
    \label{tab:tp}
\end{table}
These values correspond to the words $\varepsilon,1,10,11,100,101,110,111$.
In Table~\ref{tab:tp}, the elements of the classical Pascal triangle are written in bold. A visual representation is given in Figure \ref{fig:comp} where we have represented the positive values and a compressed version of the same figure. In Table~\ref{tab:tp}, we also add an extra column in which we count, on each row of $\mathsf{P}_2$, the number of positive binomial coefficients as explained in Definition~\ref{def:S}.    
 \begin{figure}[h!t]
      \centering
      \scalebox{.5}{\includegraphics{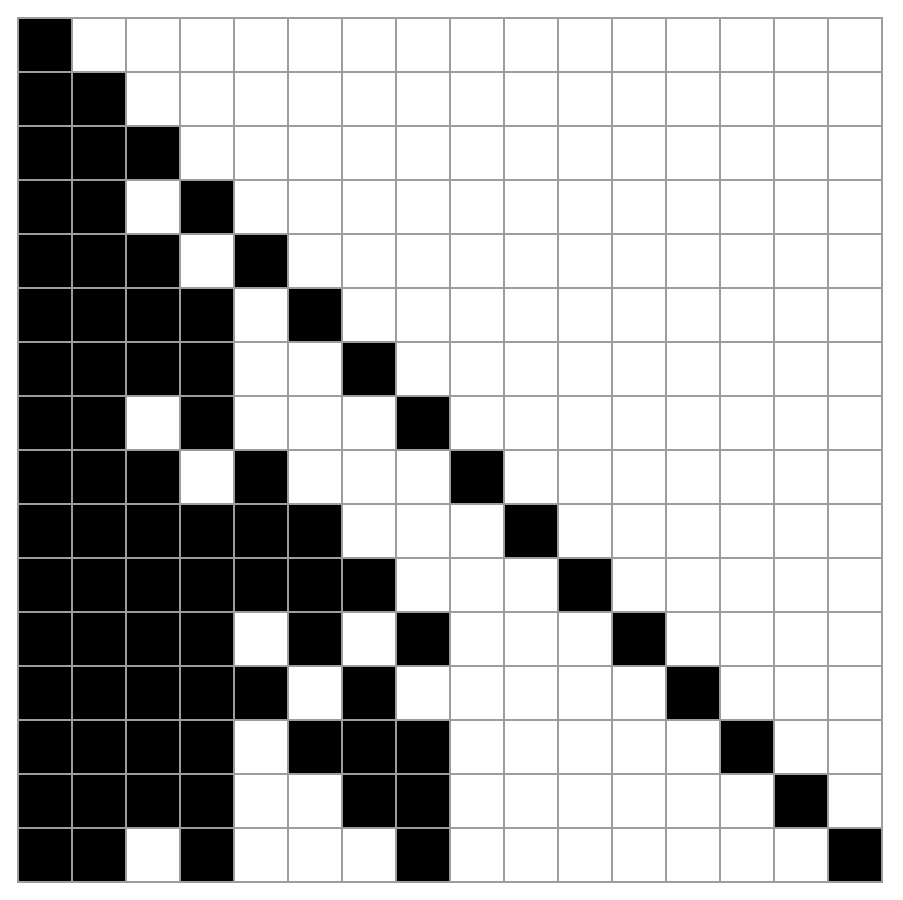}}\quad\quad \scalebox{.5}{\includegraphics{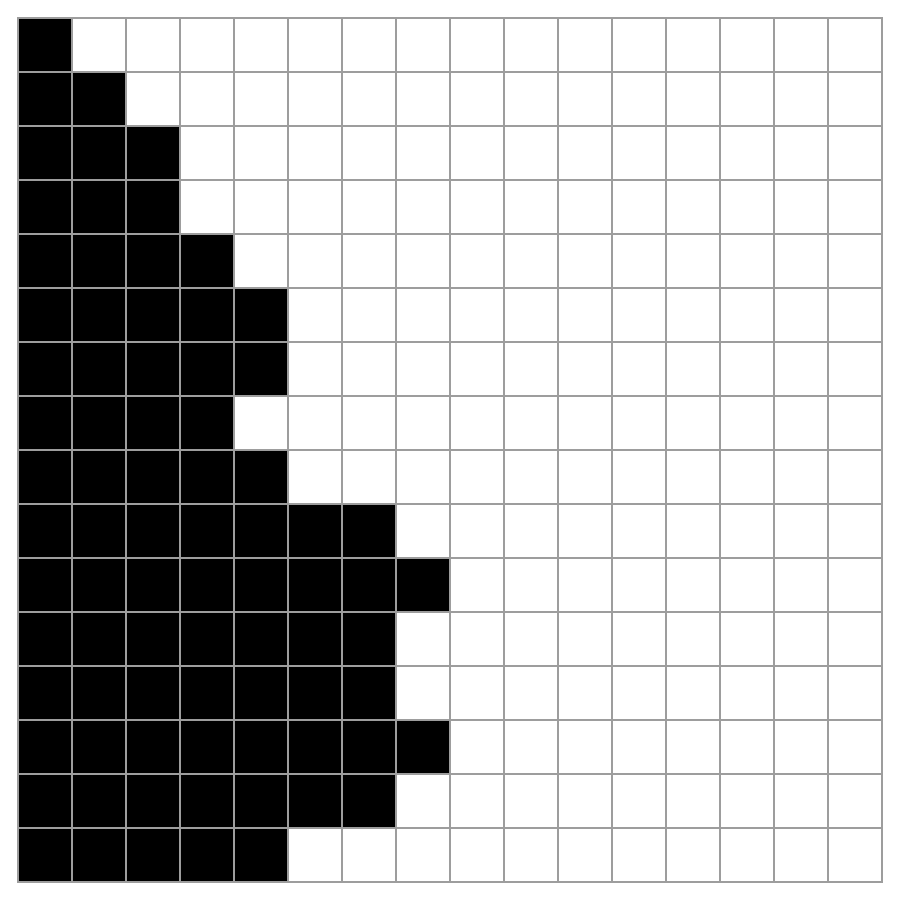}}
      \caption{Positive values in the generalized Pascal triangle $\mathsf{P}_2$ (on the left) and the compressed version (on the right).}
      \label{fig:comp}
  \end{figure}

\begin{definition}\label{def:S}
We are interested in the sequence\footnote{The sequence obtained by adding an extra $1$ as a prefix of our sequence of interest matches exactly the sequence {\tt A007306} found in Sloane's On-Line Encyclopedia of Integer Sequences \cite{EOIS}.} $(S(n))_{n\ge 0}$ whose $n$th term, $n\ge 0$, is the number of non-zero elements in the $n$th row of $\mathsf{P}_2$. Hence, the first few terms of this sequence are
$$1, 2, 3, 3, 4, 5, 5, 4, 5, 7, 8, 7, 7, 8, 7, 5, 6, 9, 11, 10, 11, 13, 12, 9, 9, 12, 13, 11, 10, \ldots.$$
Otherwise stated, for $n\ge 0$, we define
\begin{equation}
    \label{eq:defS}
    S(n):=\#\left\{v\in L_2\mid \binom{\rep_2(n)}{v}>0\right\}.
\end{equation}
\end{definition}

Observe that on a one-letter alphabet, i.e., for the classical Pascal triangle, the analogue sequence satisfies $S'(n)=n+1$ for all $n\ge 0$. 
As we can see in Figure~\ref{fig:Sn}, the sequence $(S(n))_{n \geq 0}$ has a much more chaotic behavior.
\begin{figure}[h!tbp]
    \centering
    \includegraphics[scale=1]{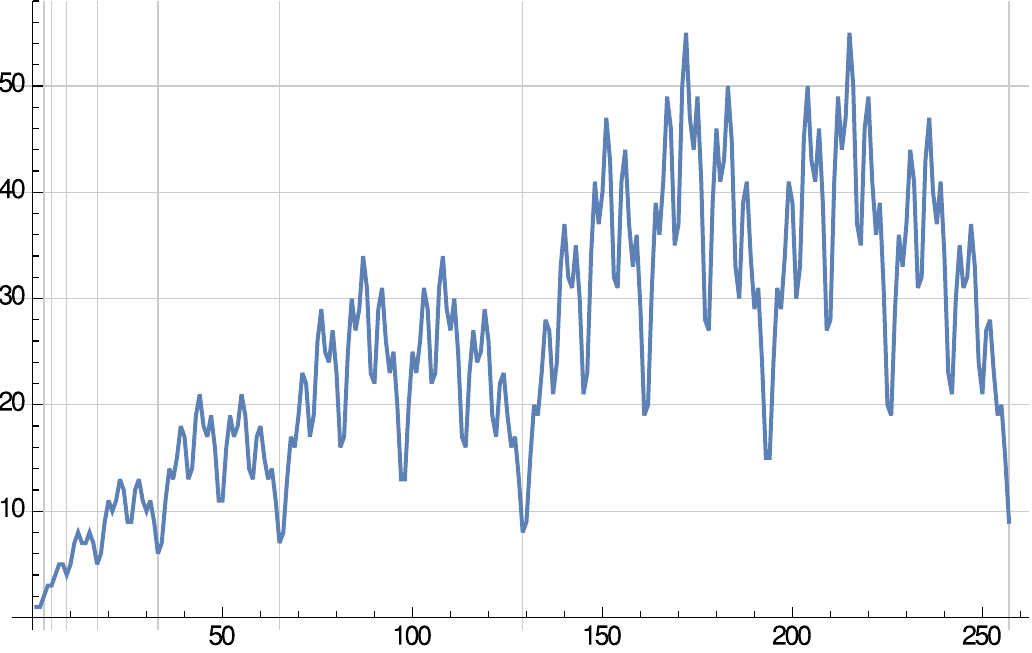}
    \caption{The sequence $(S(n))_{n\ge 0}$ in the interval $[0,256]$.}
    \label{fig:Sn}
\end{figure}

However Figure~\ref{fig:Sn} shares some similarities with the sequence studied in \cite{Elise} and independently in \cite{greinecker} about the $2$-abelian complexity of the Thue-Morse word and relative infinite words. For instance, we may observe a palindromic structure over each interval of the form $[2^n,2^{n+1})$. This suggests that $(S(n))_{n\ge 0}$, as the Stern--Brocot sequence, should be $2$-regular (see Section~\ref{sec:2reg}). 

\subsection{Our contribution} In Section~\ref{sec:3}, we introduce a convenient tree structure, called the trie of subwords, that permits us to easily count the number of subwords occurring in a given word. 
In Section~\ref{sec:proof}, based on this notion, we obtain a recurrence relation satisfied by $(S(n))_{n\ge 0}$ and stated in Proposition~\ref{pro:rec}. 
This result induces a relationship with the sequence {\tt A007306} and with the Stern--Brocot sequence $(SB(n))_{n\ge 0}$ {\tt A002487}. In Section~\ref{sec:2}, we show that $S(n)=SB(2n+1)$ for all $n\ge 0$. 
From the $2$-regularity of the Stern--Brocot sequence \cite[Example 7]{AS99}, one can immediately deduce that $(S(n))_{n\ge 0}$ is also 2-regular. 
In Section~\ref{sec:2reg}, using Proposition~\ref{pro:rec}, we propose an alternative proof of that property. 
As a corollary, $S(n)$ can be computed by multiplying $2\times 2$ matrices and the number of multiplications is proportional to $\log_2(n)$. Amongst the unbounded $2$-regular sequences, the $2$-synchronized sequences are those that can be ``computed'' with a finite automaton. 
In Section~\ref{sec:sync}, we shortly discuss the fact that $(S(n))_{n\ge 0}$ is not $2$-synchronized. 
In Section~\ref{sec:4}, we suggest that our techniques could be extended to base-$k$ numeration systems for $k \geq 3$. This gives a motivation to study further such extensions of sequences related to the Stern--Brocot sequence.

In the last part of the paper, we consider a new variation of the Pascal triangle by reducing the entries to the admissible representations of integers in the Zeckendorf numeration system based on the Fibonacci sequence. Reconsidering our initial problem with this new array leads to a new sequence denoted by $(S_F(n))_{n\ge 0}$. Surprisingly, this sequence still behaves quite well. We get a recurrence relation similar to Proposition~\ref{pro:rec}. Adapting the tries of subwords, we prove in Section~\ref{sec:freg} that $(S_F(n))_{n\ge 0}$ is $F$-regular (the definition will be recalled) in the sense of \cite{Sha88,AST} meaning that the $\mathbb{Z}$-module generated by the kernel adapted to the Zeckendorf numeration system is finitely generated. As a corollary, $S_F(n)$  can be computed by multiplying $2\times 2$ matrices and the number of multiplications is proportional to $\log_\varphi(n)$ where $\varphi$ is the golden ratio. 
Other sequences exhibiting this $F$-regularity can be found in \cite{berstel,DMSS}.

%========================================
\section{The trie of subwords}\label{sec:3}
%===============================================

In \cite{Kar}, an automaton with multiplicities accepting exactly the subwords $v$ of a given word $u$ is presented. The number of accepting paths is exactly $\binom{u}{v}$. As mentioned in \cite{KNS}, it is an important problem to determine what the ``best'' data structures for reasoning with subwords are. Also see \cite{BDS}.

For our particular needs, we first introduce a convenient tree structure, where the nodes correspond to the subwords of a given word. This tree not only leads to a recurrence relation to compute $S(n)$ directly from $\rep_2(n)$ but can also be generalized to larger alphabets and to the Fibonacci case considered in the last part of the paper.

\begin{definition}\label{def:trie}
    Let $w$ be a finite word over $\{0,1\}$. The language of its subwords is factorial, i.e., if $xyz$ is a subword of $w$, then $y$ is also a subword of $w$. Thus we may associate with $w$ the {\em trie of its subwords} denoted by $\mathcal{T}(w)$.  
The root is $\varepsilon$, and if $u$ and $ua$ are two subwords of $w$ with $a\in\{0,1\}$, then $ua$ is a child of $u$.
This trie is also called prefix tree or radix tree. All successors of a node have a common prefix. 

If we are interested in words and subwords belonging to a specific factorial language $K$, we only consider the part of $\mathcal{T}(w)$ that is a subtree of the tree associated with $K$. This subtree is denoted by $\mathcal{T}_K(w)$.
\end{definition}

In the first part of this paper, we are interested in words and subwords belonging to $L_2$. We will thus consider the tree $\mathcal{T}_{L_2}(w)$. This means that in $\mathcal{T}(w)$ we will only consider the child $1$ of the root $\varepsilon$.

\begin{remark} 
Note that the number of nodes on level $\ell\ge 0$ in $\mathcal{T}_{L_2}(w)$ counts the number of subwords of length $\ell$ in $L_2$ occurring in $w$.
In particular, the number of nodes of the trie $\mathcal{T}_{L_2}(\rep_2(n))$ is exactly $S(n)$ for all $n\ge 0$.
\end{remark}

\begin{example}\label{exa:tr1}
In Figure~\ref{fig:trie1}, we have depicted the tree $\mathcal{T}_{L_2}(11001110)$ (the dashed lines and the subtrees $T_\ell$, $\ell\in\{0,\ldots,3\}$, will become clear with Definition~\ref{def:subtree} below).
The word $w$ is highlighted as the leftmost branch (it is more visual, if we do not use the convention that the left child of $u$ is $u0$ and the right child of $u$ is $u1$). The edge between $u$ and its child $u0$ (resp., $u1$) is represented in gray (resp., black).
\begin{figure}[h!tb]
    \centering
    {\psfrag{T0}{$T_0$}\psfrag{T1}{$T_1$}\psfrag{T2}{$T_2$}\psfrag{T3}{$T_3$}
\psfrag{0}{$0$}\psfrag{1}{$1$}\includegraphics{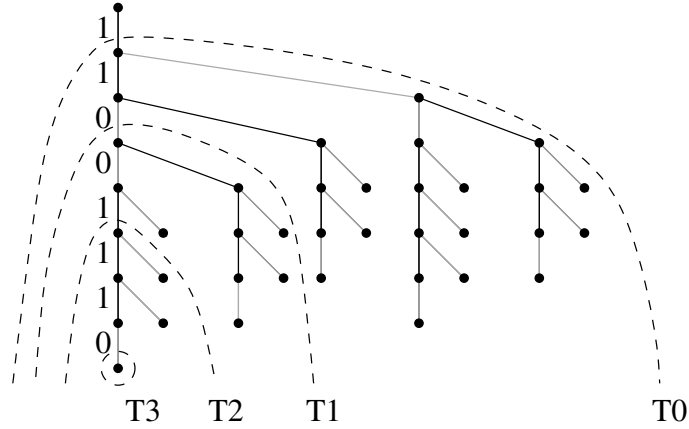}}
    \caption{The trie $\mathcal{T}_{L_2}(11001110)$.}
    \label{fig:trie1}
\end{figure}
\end{example}

Since we are dealing with subwords of $w$, the tree $\mathcal{T}_{L_2}(w)$ has a particular structure that will be helpful to count the number of distinct subwords occurring in $w$.
We now describe this structure which permits us to construct $\mathcal{T}_{L_2}(w)$ starting with a linear tree and proceeding bottom-up (see Example~\ref{exa:bu}).

\begin{definition}\label{def:subtree}
Each non-empty word $w$ in $L_2$ is factorized into consecutive maximal blocks of ones and blocks of zeroes. It is of the form
\begin{equation}
    \label{eq:factoriz}
     w=\underbrace{1^{n_1}}_{u_1} \underbrace{0^{n_2}}_{u_2} \underbrace{1^{n_3}}_{u_3} \underbrace{0^{n_4}}_{u_4} \cdots \underbrace{1^{n_{2j-1}}}_{u_{2j-1}} \underbrace{0^{n_{2j}}}_{u_{2j}}
\end{equation}
with $j\ge 1$, $n_1,\ldots, n_{2j-1}\ge 1$ and $n_{2j}\ge 0$.

Let $M=M_w$ be such that $w=u_1u_2\cdots u_M$ where $u_M$ is the last non-empty block of zeroes or ones.
For every $\ell\in\{0,\ldots,M-1\}$, the subtree of $\mathcal{T}_{L_2}(w)$ whose root is the node $u_1\cdots u_\ell 1$ (resp., $u_1\cdots u_\ell 0$) if $\ell$ is even (resp., odd) is denoted by $T_\ell$ (if $\ell=0$, we set $u_1\cdots u_\ell = \varepsilon$). For convenience (Corollary~\ref{cor:arbreTl}), we set $T_M$ to be an empty tree with no node. Roughly speaking, we have a root of a new subtree $T_\ell$ for each new alternation of digits in $w$.
\end{definition}
For the word $11001110$ of Example~\ref{exa:tr1}, we have $M=4$. In Figure~\ref{fig:trie1}, we have represented the trees $T_0,\ldots,T_3$ where $T_3$ is limited to a single node. Indeed, the number of nodes of $T_{M-1}$ is $n_M$, which is equal to $1$ in this example.

Since we are considering subwords of $w$, we immediately get the following.

\begin{prop}\label{pro:forme_arbre}
Let $w$ be a finite word in $L_2$. With the above notation about $M$ and the subtrees $T_\ell$,
the tree $\mathcal{T}_{L_2}(w)$ has the following properties.
\begin{itemize}
\item
Assume that $2\le 2k<M$.
For every $j\in\{0,\ldots,n_{2k}-1\}$, the node of label $x=u_1 \cdots u_{2k-1} 0^j$ has two children $x0$ and $x1$.
The node $x1$ is the root of a tree isomorphic to $T_{2k}$.

\item
Assume that $3\le 2k+1<M$.
Similarly, for every $j\in\{0,\ldots,n_{2k+1}-1\}$, the node of label $x=u_1 \cdots u_{2k} 1^j$ has two children $x0$ and $x1$.
The node $x0$ is the root of a tree isomorphic to $T_{2k+1}$. 
\item
For every $j\in\{1,\ldots,n_{1}-1\}$, the node of label $x=1^j$ has two children $x0$ and $x1$. The node $x0$ is the root of a tree isomorphic to $T_{1}$.
\end{itemize}
\end{prop}

As depicted in Figure~\ref{fig:trie2}, this proposition permits us to reconstruct the tree $\mathcal{T}_{L_2}(w)$ from $w$.

\begin{example}\label{exa:bu}
In Figure~\ref{fig:trie2}, we show how to build $\mathcal{T}_{L_2}(11001110)$ in four steps. In this example, we have 
\[
     11001110 = \underbrace{1^{2}}_{u_1} \underbrace{0^{2}}_{u_2} \underbrace{1^{3}}_{u_3} \underbrace{0^{1}}_{u_4}.
\]
Hence the tree $T_3$ is limited to a single node because $n_4=1$. Actually, $T_{M-1}$ is a linear tree with $n_M$ nodes. 

To reconstruct $\mathcal{T}_{L_2}(11001110)$, start with a linear tree corresponding to the word (it is depicted on the left in Figure~\ref{fig:trie2}). 
We proceed from the bottom of the tree. 
%First, the tree $T_3$ is the linear subtree consisting in the last $n_4 = 1$ node.
First, add a copy of $T_3$ to each node of the form $u_1 u_2 1^j$ for $j=0,1,2$ (second picture). 

Then consider the subtree $T_2$ whose root is the node $u_1\cdots u_2 1$ and add a copy of it to each node of the form $u_1 0^j$ for $j=0,1$ (third picture). 

Finally, consider the subtree $T_1$ whose root is the node $u_10$ and add a copy of it to the node $1$ (picture on the bottom). 
Note that if $u_1=1^\ell$, we should add a copy of $T_1$ to each node of the form $1^j$ for $j=1,\ldots,\ell-1$ (if $\ell=1$, then no copy of $T_1$ is added).
 \begin{figure}[h!tb]
    \centering
    {\psfrag{T0}{$T_0$}\psfrag{T1}{$T_1$}\psfrag{T2}{$T_2$}\psfrag{T3}{$T_3$}
\psfrag{0}{$0$}\psfrag{1}{$1$}\includegraphics{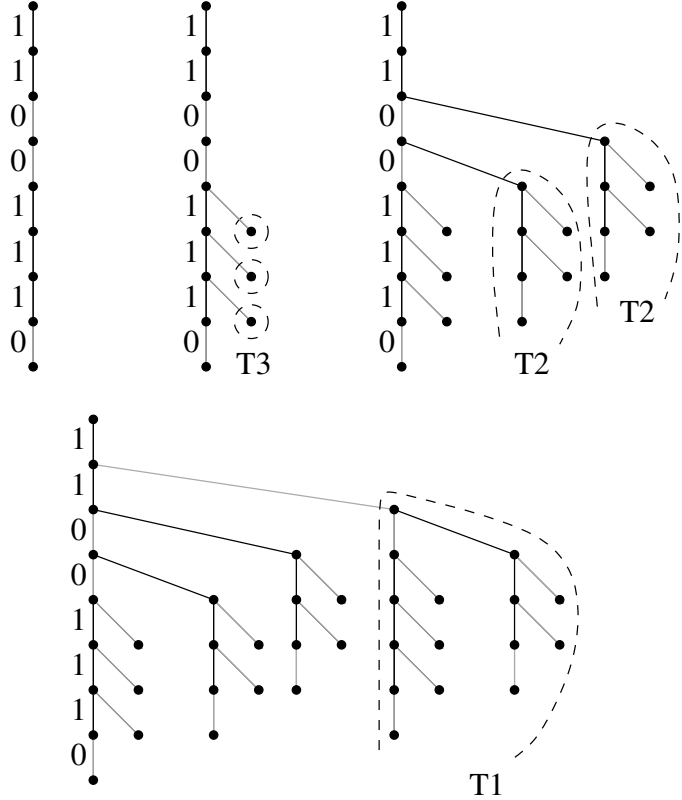}}
    \caption{Bottom-up construction of $\mathcal{T}_{L_2}(11001110)$.}
    \label{fig:trie2}
\end{figure}   
\end{example}

For the next statement, recall that $T_M$ is an empty tree. If $T$ is a tree, we let $\# T$ denote the number of nodes in $T$. This statement gives a recurrence relation to compute $S(n)$ directly from $\rep_2(n)$.
\begin{corollary}\label{cor:arbreTl}
Let $w$ be a finite word in $L_2$. With the above notation about $M$ and the subtrees $T_\ell$, the following holds.
The number of nodes in $T_{M-1}$ is $n_M$.
For $j=M-2,M-3,\ldots,0$, the number of nodes in $T_j$ is given by
$$\# T_j=n_{j+1}(\# T_{j+1}+1)+\# T_{j+2}.$$
The number of subwords in $L_2$ of $w$ is given by $1+\# T_0$.
\end{corollary}

\begin{proof}
    The idea is the same as the one developed in the previous example. Start with a linear tree labeled by $w$ and add, with a bottom-up approach, all the possible subtrees given by Proposition~\ref{pro:forme_arbre}: first possible copies of $T_{M-1}$, then $T_{M-2}$, \ldots, $T_1$.
\end{proof}

\begin{remark}
    If we were interested in the number of subwords of $w$ (not only those in $L_2$), we must add the node $0$ which will be the root of a subtree equal to $T_1$. Thus, the total number of subwords occurring in $w$ is $1+\# T_0+\# T_1$.
\end{remark}

\begin{example}
    For the word $w=11001110$, $n_4=1$, $\# T_3=1$ and $\# T_4=0$. Thus, 
$\# T_2= 3 (1+1)+0=6$, $\# T_1= 2 (6+1)+1=15$, $\# T_0= 2 (15+1)+6=38$. We get that the number of subwords of $w$ in $L_2$ is $39$ and since, $\val_2(11001110)=206$, $S(206)=39$. Moreover, the total number of subwords of $11001110$ is $39+15=54$.
\end{example}

%===============================================
\section{A relation satisfied by the sequence $(S(n))_{n\ge 0}$}\label{sec:proof}
%===============================================

Thanks to the trie of subwords introduced in the previous section, we collect several results about the number of words occurring as subwords of words with a prescribed form. The aim of this section is to prove the following result. 

\begin{prop}\label{pro:rec}
    For all $\ell\ge 1$ and $0\le r < 2^\ell$, the sequence $(S(n))_{n\ge 0}$ given in Definition~\ref{def:S} satisfies $S(0)=1$ and $S(1)=2$ and 
\begin{equation}
    \label{eq:rec}
S(2^{\ell}+r)=\left\{
    \begin{array}{ll}
        S(2^{\ell-1}+r)+S(r),& \text{ if } 0 \le r < 2^{\ell-1};\\
        S(2^{\ell+1}-r-1),& \text{ if } 2^{\ell-1} \le r < 2^\ell.\\
    \end{array}\right.
\end{equation}
\end{prop}

This will permit us to show that the sequence $(S(n))_{n\ge 0}$ is the Farey sequence (see Section~\ref{sec:2}).

\begin{lemma}\label{lem:palindrome}
Let $u$ be a word in $\{0,1\}^*$. Define $\underline{u}$ by replacing in $u$ every $0$ by $1$ and every $1$ by $0$. Then
$$
\#\left\{ v\in L_2 \mid \binom{1u}{v} > 0 \right\} = \#\left\{ v\in L_2 \mid \binom{1\underline{u}}{v} > 0 \right\}.
$$
In particular, this means that $S(2^{\ell}+r)=S(2^{\ell+1}-r-1)$, if $2^{\ell-1} \le r < 2^\ell$.
\end{lemma}

\begin{proof}
    It is enough to observe that the trees $\mathcal{T}_{L_2}(1u)$ and $\mathcal{T}_{L_2}(1\underline{u})$ are isomorphic. Each node of the form $1x$ in the first tree corresponds to the node $1\underline{x}$ in the second one and conversely.

For the special case, note that for every word $z$ of length $\ell$, there exists $r\in\{0,\ldots,2^\ell-1\}$ such that $\rep_2(2^\ell +r)=1z$ and
$$\val_2(\underline{z})=2^\ell-r-1 \in\{0\ldots,2^\ell-1\}.$$
Hence, $1\underline{z}=\rep_2(2^{\ell+1}-r-1)$. Using \eqref{eq:defS}, we obtain the desired result.
\end{proof}

\begin{lemma}\label{lem:u=100}
Let $u$ be a word in $\{0,1\}^*$. Then 
$$\#\left\{ v\in L_2 \mid \binom{100u}{v} > 0 \right\} = 2 \cdot \#\left\{ v\in L_2 \mid \binom{10u}{v} > 0 \right\} - \#\left\{ v\in L_2 \mid \binom{1u}{v} > 0 \right\}.$$
\end{lemma}

\begin{proof}
    Our reasoning is again based on the structure of the trees. Assume first that $u$ has no $1$, then $u=0^n$, $n\ge 0$. The tree
$\mathcal{T}_{L_2}(1u)$ is linear and has $n+2$ nodes, $\mathcal{T}_{L_2}(10u)$ has $n+3$ nodes and $\mathcal{T}_{L_2}(100u)$ has $n+4$ nodes. The formula holds.

Now assume that $u$ has a $1$. First observe that the subtree $S$ of $\mathcal{T}_{L_2}(1u)$ with root $1$ is equal to the subtree of $\mathcal{T}_{L_2}(10u)$ with root $10$ and also to the subtree of $\mathcal{T}_{L_2}(100u)$ with root $100$. Consider the shortest prefix of $1u$ of the form $10^r1$. Let $R$ be the subtree of $\mathcal{T}_{L_2}(1u)$ with root $10^r1$. The subtree of $\mathcal{T}_{L_2}(10u)$ with root $11$ is $R$. Similarly, $\mathcal{T}_{L_2}(100u)$ contains two copies of $R$: the subtrees of root $11$ and $101$. The situation is depicted in Figure~\ref{fig:RS} and the following formula holds: 
$3+\# S+2\# R=2(2+\# S+\# R)-(1+\# S)$.
\begin{figure}[h!tb]
    \centering
    {\psfrag{R}{$R$}\psfrag{S}{$S$}\psfrag{0}{$0$}\psfrag{1}{$1$}\includegraphics{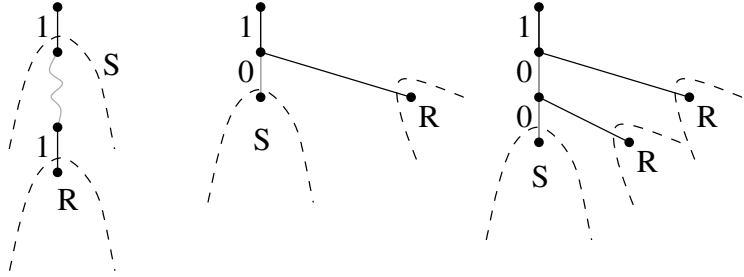}}
    \caption{Structure of the trees.}
    \label{fig:RS}
\end{figure}
\end{proof}

\begin{lemma}\label{lem:u=101}
Let $u$ be a word in $\{0,1\}^*$. Then 
$$\#\left\{ v\in L_2 \mid \binom{101u}{v} > 0 \right\} = \#\left\{ v\in L_2 \mid \binom{1u}{v} > 0 \right\} + \#\left\{ v\in L_2 \mid \binom{11u}{v} > 0 \right\}.$$
\end{lemma}

\begin{proof}
    The reasoning is similar to the one of the previous proof. If $u$ is of the form $1^n$, then $\mathcal{T}_{L_2}(101u)$ has $2n+5$ nodes and $\mathcal{T}_{L_2}(1u)$ and $\mathcal{T}_{L_2}(11u)$  have respectively $n+2$ and $n+3$ nodes. If $u$ has a $0$, then consider the shortest prefix of $101u$ of the form $101^r0$. Let $S$ be the subtree of $\mathcal{T}_{L_2}(101u)$ with root $101$ and $R$ be the subtree with root $101^r0$. The tree $\mathcal{T}_{L_2}(101u)$ (resp., $\mathcal{T}_{L_2}(1u)$, $\mathcal{T}_{L_2}(11u)$) has $3+2\# S+\# R$ (resp., $1+\# S$, $2+\# S+\# R$) nodes.
\end{proof}

\begin{proof}[Proof of Proposition~\ref{pro:rec}]
Consider some integers $\ell$ and $r$ with $\ell \geq 1$ and $0 \leq r < 2^{\ell}$.
If $2^{\ell-1} \le r < 2^\ell$, the equality $S(2^\ell+r) = S(2^{\ell+1}-r-1)$ directly follows from Lemma~\ref{lem:palindrome}.

So let us suppose $0 \leq r < 2^{\ell-1}$.
We proceed by induction on $\ell$.
The case $\ell = 1$ is easily checked. Let us suppose $\ell \geq 2$.
By definition we have
\[
    S(2^\ell+r)= \#\left\{v\in L_2\mid \binom{\rep_2(2^\ell+r)}{v}>0\right\}.
\]
Since $r < 2^{\ell-1}$, we have $\rep_2(2^\ell+r) = 10u$ for some word $u \in \{0,1\}^*$ of length $\ell-1$ such that $r=\val_2(u)$. We consider two cases depending on the first letter occurring in $u$.

If $u \in 0 \{0,1\}^*$, then $r=\val_2(u) \le 2^{\ell-2}-1$ and we deduce from Lemma~\ref{lem:u=100} that 
$$S(2^\ell+r)= 2 S(2^{\ell-1}+r) - S(2^{\ell-2}+r).$$
The proof is complete after using the induction hypothesis twice:
\begin{eqnarray*}
    S(2^\ell+r)&=&2(S(2^{\ell-2}+r)+S(r))-S(2^{\ell-2}+r)\\
&=& S(2^{\ell-2}+r)+S(r)+S(r)\\
&=&S(2^{\ell-1}+r)+S(r).
\end{eqnarray*}

If $u=1u'$ for some $u'\in\{0,1\}^*$, then 
$\val_2(101u')=2^\ell+r$, $\val_2(1u')=\val_2(u)=r$ and $\val_2(11u')=2^{\ell-1}+r$. We deduce from Lemma~\ref{lem:u=101} applied to $u'$ that 
\[
	S(2^\ell+r) = S(r) + S(2^{\ell-1}+r),
\]
which finishes the proof.
\end{proof}

%===============================================
\section{The sequence {\tt A007306}, Farey tree and Stern--Brocot sequence}\label{sec:2}
%===============================================

Plugging in the first few terms of $(S(n))_{n\ge 0}$ in Sloane's On-Line Encyclopedia of Integer Sequences \cite{EOIS}, it seems to be a shifted version of the sequence {\tt A007306} of the denominators occurring in the Farey tree (left subtree of the full Stern--Brocot tree), which contains every (reduced) positive rational less than $1$ exactly once. Many papers deal with this tree; for instance, see \cite{Bates,Glasby}. 
The {\em Farey tree} is an infinite binary tree made up of mediants. 
Given two reduced fractions $\frac{a}{b}$ and $\frac{c}{d}$, with $a,b,c,d \in \mathbb{N}$, their {\em mediant} is the fraction $\frac{a}{b} \oplus \frac{c}{d} = \frac{a+c}{b+d}$.
This operation is known as {\em child's addition}.
Observe that for all $\frac{a}{b}$, $\frac{c}{d}$ with $\frac{a}{b} < \frac{c}{d}$, we have $\frac{a}{b} < \frac{a}{b} \oplus \frac{c}{d} < \frac{c}{d}$.

\begin{definition}
Starting from the fractions $\frac{0}{1}$ and $\frac{1}{1}$, the {\em Farey tree} is the infinite tree defined as follows:
\begin{enumerate}
\item 
the set of nodes is partitioned into levels indexed by $\mathbb{N}$;
\item
the level $0$ consists in $\{\frac{0}{1},\frac{1}{1}\}$;
\item
the level $1$ consists in $\{\frac{1}{2}\}$. The node $\frac12$ is the only one with two parents, which are $\frac{0}{1}$ and $\frac{1}{1}$;
\item
for each $n \geq 2$, the level $n$ consists in the children of the nodes of vertices in level $n-1$.
For each node $\frac{a}{b}$ of level $n-1$, we define 
\begin{eqnarray*}
	\mathrm{Left}\left(\frac{a}{b}\right) 
	&=& \max\left\{ \frac{e}{f} \mid \mathrm{level}\left(\frac{e}{f}\right) < n-1 \text{ and } \frac{e}{f}<\frac{a}{b} \right\};\\
	\mathrm{Right}\left(\frac{a}{b}\right) 
	&=& \min\left\{ \frac{e}{f} \mid \mathrm{level}\left(\frac{e}{f}\right) < n-1 \text{ and } \frac{e}{f} > \frac{a}{b} \right\}.	
\end{eqnarray*}
The left and right children of $\frac{a}{b}$ are respectively $\frac{a}{b} \oplus \mathrm{Left}\left( \frac{a}{b} \right)$ and $\frac{a}{b} \oplus \mathrm{Right}\left( \frac{a}{b} \right)$.
\end{enumerate}
\end{definition}

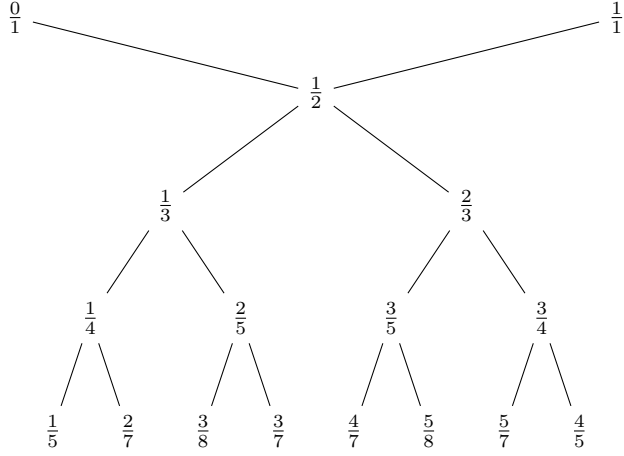
\begin{figure}[h!tb]
    \centering
\begin{tikzpicture}
\node(a) at (-4,1) {$\frac01$};
\node(b) at (4,1) {$\frac11$};
[-,thick]
\node(c) at (0,0) {$\frac12$}
  [sibling distance=4cm]
  child {node {$\frac13$}
   [sibling distance=2cm]
   child {node {$\frac14$}
    [sibling distance=1cm]
    child {node {$\frac15$}}
    child {node {$\frac27$}}
  }
   child {node {$\frac25$}
    [sibling distance=1cm] 
    child {node {$\frac38$}}
    child {node {$\frac37$}}
   }
  }
  child {node {$\frac23$}
   [sibling distance=2cm]
   child {node {$\frac35$}
    [sibling distance=1cm]
    child {node {$\frac47$}}
    child {node {$\frac58$}}
  }
   child {node {$\frac34$}
    [sibling distance=1cm]
    child {node {$\frac57$}}
    child {node {$\frac45$}}
   }
  };
\draw [-, >=latex] (a) to node {} (c);
\draw [-, >=latex] (b) to node {} (c);
\end{tikzpicture}
    \caption{The first levels of the Farey tree.}
    \label{fig:farey}
\end{figure}
Reading the denominators in Figure~\ref{fig:farey} level-by-level, then from left to right, i.e., conducting a breadth-first traversal of the tree, we obtain the first few terms of the sequence $(S(n))_{n\ge 0}$ if we drop the second denominator $1$. 
If a branch to the left (resp., right) corresponds to $0$ (resp., $1$), then with every node $\frac{a}{b}$ (except $\frac11$) is associated a unique path from $\frac01$ to $\frac{a}{b}$ of label $u$. If $n=\val_2(u)$, then we will show that $b=S(n)$. For instance, $\frac37$ corresponds to the path of label $1011 = \rep_2(11)$ and $S(11)=7$. Conversely, given an integer $n\ge 0$, we let $D(n)$ be the denominator of the node reached from $\frac01$ using the path of label $\rep_2(n)$. The sequence $(D(n))_{n\ge 0}$ is indexed by {\tt A007306} in \cite{EOIS}.

We show that the sequence $(D(n))_{n\ge 0}$ satisfies $D(n) = S(n)$ for all $n$.
From the definition of the tree, for all $\ell\ge 1$, $u\in\{0,1\}^*$, it directly follows that
$$
	D(\val_2(u10^\ell))=D(\val_2(u10^{\ell-1}))+D(\val_2(u))
$$
and
$$
	D(\val_2(u01^\ell))=D(\val_2(u01^{\ell-1}))+D(\val_2(u)).
$$

Otherwise stated, if $n = 2^k +r$ with $0 \leq r < 2^k$, we have
\[
	D(n) =
	\begin{cases}
		D(n/2) + D((n-2^\ell)/2^{\ell+1}), 
		& \text{if } \rep_2(r) \in \{0,1\}^* 10^\ell; \\
		D((n-1)/2) + D((n-2^\ell+1)/2^{\ell+1}), 
		& \text{if } \rep_2(r) \in \{0,1\}^* 01^\ell. \\
	\end{cases}
\]
We show by induction that the sequence $(S(n))_{n \geq 0}$ satisfies the same formulas, i.e., 
\[
	S(n) =
	\begin{cases}
		S(n/2) + S((n-2^\ell)/2^{\ell+1}), 
		& \text{if } \rep_2(r) \in \{0,1\}^* 10^\ell; \\
		S((n-1)/2) + S((n-2^\ell+1)/2^{\ell+1}), 
		& \text{if } \rep_2(r) \in \{0,1\}^* 01^\ell. \\
	\end{cases}
\]
Note that a direct inspection shows that $S(n) = D(n)$, $0\le n\le 3$.
For the induction step, we write $n = 2^k +r$ with $k\ge 2$ and $0 \leq r < 2^k$. Using~\eqref{eq:rec}, we then consider four cases, depending on whether $0 \leq r < 2^{k-1}$ or $2^{k-1} \leq r < 2^k$ and whether $\rep_2(r)$ has a suffix consisting of zeroes or of ones.
We only give the proof for the case $0 \leq r < 2^{k-1}$ with $\rep_2(r) \in \{0,1\}^*10^\ell$, $\ell \geq 1$, the other ones being similar.
We have
\[
\begin{array}{rclr}
	S(n)
	&=& S(2^{k-1}+r) + S(r) 
	& \text{(by \eqref{eq:rec})} \\
	&=& S((2^{k-1}+r)/2) + S((2^{k-1}+r-2^\ell)/2^{\ell+1}) \\
	& & + S(r/2) + S((r-2^\ell)/2^{\ell+1}) 
	& \text{(by induction hypothesis)}  \\
	&=& S(2^{k-1}+r/2) + S(2^{k-\ell-1}+r/2^{\ell+1}-1/2) 
	& \text{(by \eqref{eq:rec})} 		\\
	&=& S(n/2) + S((n-2^\ell)/2^{\ell+1}).
\end{array}
\]

\begin{remark}
    It is a folklore fact that the sum of the denominators of the elements of the $k$th level (with $k\ge 1$) in the Farey tree is equal to $2\cdot 3^{k-1}$ and is $1$ if we only consider the denominator $D(0)$. Or similarly, that the sum of the denominators of the levels $1$ to $k$ is equal to $3^k-1$ (or $3^k$ if we add the denominator $D(0)$ on the $0$th level). Using \eqref{eq:defS}, we observe that $\sum_{i=0}^{2^n-1} S(i)$ is the number of pairs of words in $L_n$ having a positive binomial coefficient. In \cite{LRS}, we showed that this number is equal to $3^n$, giving an alternative proof of the above mentioned fact about the Farey tree.
\end{remark}

The {\em Stern--Brocot sequence} $(SB(n))_{n\ge 0}$ is defined by $SB(0)=0, SB(1)=1$ and, for all $n\ge 1$, 
$$SB(2n)=SB(n), \quad SB(2n+1)=SB(n)+SB(n+1).$$  
It is well known that $D(n)=SB(2n+1)$. Hence the sequence $(S(n))_{n\ge 0}$ satisfies 
\begin{equation}\label{eq:SB}
S(n)=SB(2n+1) \quad \text{for all } n\ge 0.
\end{equation} 

\begin{remark}
In \cite{CW}, it is shown that the $n$th Stern--Brocot value $SB(n)$ is equal to the number of times words of the form $v\in 1(01)^*$ occur as scattered subwords of the binary expansion of $n$. This result is different from the one obtained here because the form of the subwords is fixed.
In \cite{CoSha}, the authors give a way to build the sequence $(SB(n))_{n\ge 0}$ using occurrences of words occurring in the base-$2$ expansions of positive integers.
\end{remark}

%==================================
\section{The sequence $(S(n))_{n\ge 0}$ is $2$-regular}\label{sec:2reg}
%===============================================

The $k$-regularity of a sequence provides interesting structural information about it. 
For instance, we get matrices to compute its $n$-th term in a number of steps proportional to $\log_k(n)$.
First we recall the notions of $k$-kernel and $k$-regular sequence. Since $(SB(n))_{n\ge 0}$ is $2$-regular \cite{AS99}, one can immediately deduce that the sequence $(S(n))_{n\ge 0}$ is $2$-regular from~\eqref{eq:SB}. Nevertheless, we provide an alternative proof because we have in mind extensions to other numeration systems; see Section~\ref{sec:4} and Section~\ref{sec:fibo}.

\begin{definition}\label{def:ker}
Let $k\geq 2$ be an integer. The {\em $k$-kernel} of a sequence $s=(s(n))_{n\geq 0}$ is the set
$$\mathcal{K}_k(s) = \{ (s(k^in + j))_{n\ge 0} | \; i\ge 0 \text{ and } 0\le j < k^i \}.
$$
\end{definition}

One characterization of $k$-automatic sequences is that their $k$-kernels are finite; see \cite{E74} or \cite{AS03}. For instance, the $2$-kernel $\mathcal{K}_2(\mathbf{t})$ of the Thue--Morse word $\mathbf{t}=01101001\cdots$ contains exactly two elements, namely $\mathbf{t}$ and $\underline{\mathbf{t}}$. Unbounded sequences (i.e., taking infinitely many integer values) are also of interest but clearly, their $k$-kernels are infinite. 
To handle such sequences, one introduces the following definition in the sense of Allouche and Shallit \cite{AS99}.
Also see \cite{berreu}.

\begin{definition}\label{def:k-reg}
Let $k\ge 2$ be an integer. A sequence $s=(s(n))_{n\geq 0}$ is \emph{$k$-regular} if $\left\langle \mathcal{K}_k(s)\right\rangle$ is a finitely-generated $\mathbb{Z}$-module, i.e., there exists a finite number of sequences $(t_1(n))_{n\geq 0}, \ldots, (t_\ell(n))_{n\geq 0}$ such that every sequence in the $\mathbb{Z}$-module generated by the $k$-kernel $\mathcal{K}_k(s)$ is a $\mathbb{Z}$-linear combination of the $t_r$'s. Otherwise stated, for all $i\ge 0$ and for all $j\in\{0,\ldots, k^i - 1\}$, there
exist integers $c_1, \ldots, c_\ell$ such that
$$\forall n \ge 0, \quad s(k^i n + j) = \sum_{r=1}^\ell 
c_r\, t_r (n).
$$
\end{definition}

As an easy example of $2$-regular (unbounded) sequence, consider the Stern--Brocot sequence. The $\mathbb{Z}$-module generated by the $2$-kernel of $(SB(n))_{n\ge 0}$ is simply generated by the sequence itself and the shifted sequence $(SB(n+1))_{n\ge 0}$.

\begin{theorem}\label{the:rel}
The sequence $(S(n))_{n \ge 0}$ satisfies, for all $n \geq 0$,
\begin{eqnarray*}
	S(2n+1) 	&=& 3\, S(n)-S(2n)			\\
	S(4n) 	&=& 2\, S(2n)-S(n)		\\
	S(4n+2) &=& 4\, S(n)-S(2n).
\end{eqnarray*}
In particular, $(S(n))_{n\ge 0}$ is $2$-regular.
\end{theorem}
\begin{proof}
We first prove $S(2n+1)+S(2n)=3S(n)$. We proceed by induction on $n$. 
It can be checked by hand that the result holds true for $n \leq 1$. 
Thus consider $n > 1$ and suppose that the relation holds true for all $m < n$. We write $n = 2^\ell +r$ with $\ell \geq 1$ and $0 \leq r < 2^\ell$. We divide the proof in two according to the position of $r$ inside the interval $[0,2^\ell)$.
\begin{enumerate}
\item[(a)]
If $0 \leq r < 2^{\ell-1}$, we get
\[
\begin{array}{rclr}
	S(2n+1) + S(2n)
	&=& S(2^{\ell+1}+2r+1) + S(2^{\ell+1}+2r) & \\
	&=& S(2^{\ell}+2r+1) + S(2r+1) + S(2^{\ell}+2r) + S(2r) 
	& \text{(by \eqref{eq:rec})}	\\
	&=& 3 S(2^{\ell-1}+r) + 3 S(r)  
	&\text{(by induction hypothesis)}\\
	&=& 3 S(n)  
	&\text{(by \eqref{eq:rec}).}
\end{array}
\]
\item[(b)]
If $2^{\ell-1} \le r < 2^\ell$, we get
$$
\begin{array}{rclr}
	S(2n+1) + S(2n)
	&=& S(2^{\ell+1}+2r+1) + S(2^{\ell+1}+2r) & \\
	&=& S(2^{\ell+2}-2r-2) + S(2^{\ell+2}-2r-1) 
	&\text{(by \eqref{eq:rec})} \\
	&=& 3S(2^{\ell+1}-r-1)
	&\text{(by induction hypothesis)}\\
	&=& 3S(n)
	&\text{(by \eqref{eq:rec}).} 
\end{array}
$$
\end{enumerate}
Now we prove the remaining two relations. We again proceed by induction on $n$. It can be checked by hand that the result holds true for $n \leq 1$. Thus consider $n > 1$ and suppose that the two relations hold true for all $m < n$.
\begin{enumerate}
\item
Let us prove $S(4n) = 2 S(2n)-S(n)$.
We write $n = 2^\ell +r$ with $\ell \geq 1$ and $0 \leq r < 2^\ell$.
\begin{enumerate}
\item
If $0 \leq r < 2^{\ell-1}$, we get
\[
\begin{array}{rclr}
	S(4n) 
	&=& S(2^{\ell+2}+4r) & \\
	&=& S(2^{\ell+1}+4r) + S(4r) 
	& \text{(by \eqref{eq:rec})}	\\
	&=& 2 S(2^{\ell}+2r) - S(2^{\ell-1} + r) + 2 S(2r) - S(r) 
	&\text{(by induction hypothesis)}\\
	&=& 2 S(2n) - S(n) 
	&\text{(by \eqref{eq:rec}).}
\end{array}
\]
\item
If $2^{\ell-1} \le r < 2^\ell$, we get
$$
\begin{array}{rclr}
	S(4n) 
	&=& S(2^{\ell+2}+4r) & \\
	&=& S(2^{\ell+3}-4r-1) 
	&\text{(by \eqref{eq:rec})} \\
	&=& 3S(2^{\ell+2}-2r-1) - S(2^{\ell+3}-4r-2)
	&\text{(using the first relation)}\\
	&=& 3S(2n) - S(2^{\ell+3}-4(r+1)+2)
	&\text{(by \eqref{eq:rec})}\\
	&=& 3S(2n) - 4S(2^{\ell+1}-r-1) + S(2^{\ell+2}-2r-2) 
	&\text{(by induction hypothesis)}\\
	&=& 3S(2n) - 4S(n) + S(2n+1)
	&\text{(by \eqref{eq:rec})}\\
	&=& 2S(2n) - S(n)
	&\text{(using the first relation).}\\
\end{array}
$$
\end{enumerate}
%\item
%Now let us prove $S(4n+1) = 2 S(2n+1)-S(n+1)$. 
%Observe that since we have proved $S(4n) = 2 S(2n)-S(n)$, it is part of the induction hypothesis.
%Let us write $n = 2^\ell + r$ with $\ell \geq 1$ and $0 \leq r < 2^\ell$. 
%\begin{enumerate}
%\item
%If $0 \leq r < 2^{\ell-1}$, we get
%\[
%\begin{array}{rclr}
%	S(4n+1) 
%	&=& S(2^{\ell+2}+4r+1) & \\
%	&=& S(2^{\ell+1}+4r+1) + S(4r+1) 
%	& \text{(by \eqref{eq:rec})}	\\
%	&=& 2 S(2^{\ell}+2r+1) - S(2^{\ell-1} + r+1) + 2 S(2r+1) - S(r+1) 
%	&\text{(by induction hypothesis)}\\
%	&=& 2 S(2n+1) - S(n+1) 
%	&\text{(by \eqref{eq:rec})}. 
%\end{array}
%\]
%\item
%If $2^{\ell-1} \leq r < 2^\ell$, we get
%\[
%\begin{array}{rclr}
%	S(4n+1) 
%	&=& S(2^{\ell+2}+4r+1) & \\
%	&=& S(2^{\ell+3}-4r) 
%	&\text{(by \eqref{eq:rec})} \\
%	&=& 2 S(2^{\ell+2}-2r) - S(2^{\ell+1} - r)
%	&\text{(by induction hypothesis)}\\
%	&=& 2 S(2n+1) - S(n+1) 
%	&\text{(by \eqref{eq:rec})}. 
%\end{array}
%\]
%\end{enumerate}
\item
Finally, let us prove $S(4n+2) = 4 S(n)-S(2n)$.
Let us write $n = 2^\ell + r$ with $\ell \geq 1$ and $0 \leq r < 2^\ell$. 
\begin{enumerate}
\item
If $0 \leq r < 2^{\ell-1}$, we get
\[
\begin{array}{rclr}
	S(4n+2) 
	&=& S(2^{\ell+2}+4r+2) & \\
	&=& S(2^{\ell+1}+4r+2) + S(4r+2) 
	& \text{(by \eqref{eq:rec})}	\\
	&=& 4 S(2^{\ell-1}+r) - S(2^{\ell} + 2r) + 4 S(r) - S(2r) 
	&\text{(by induction hypothesis)}\\
	&=& 4 S(n) - S(2n) 
	&\text{(by \eqref{eq:rec})}.
\end{array}
\]
\item
If $2^{\ell-1} \leq r < 2^\ell$, we get
\[
\begin{array}{rclr}
	S(4n+2) 
	&=& S(2^{\ell+2}+4r+2) & \\
	&=& S(2^{\ell+3}-4r-3) 
	&\text{(by \eqref{eq:rec})} \\
	&=& 3S(2^{\ell+2}-2r-2) - S(2^{\ell+3}-4r-4)
	&\text{(using the first relation)}\\
	&=& 3S(2n+1) - S(2^{\ell+3}-4(r+1))
	&\text{(by \eqref{eq:rec})}\\
	&=& 3S(2n+1) - 2S(2^{\ell+2}-2r-2) + S(2^{\ell+1}-r-1) 
	&\text{(by induction hypothesis)}\\
	&=& 3S(2n+1) - 2S(2n+1) + S(n)
	&\text{(by \eqref{eq:rec})}\\
	&=& 4S(n) - S(2n)
	&\text{(using the first relation).}\\
\end{array}
\]
\end{enumerate}
%\item
%Finally, let us prove $S(4n+3) = 4 S(n+1)-S(2n+1)$.
%Let us write $n = 2^\ell + r$ with $\ell \geq 1$ and $0 \leq r < 2^\ell$. 
%\begin{enumerate}
%\item
%If $0 \leq r < 2^{\ell-1}$, we get
%\[
%\begin{array}{rclr}
%	S(4n+3) 
%	&=& S(2^{\ell+2}+4r+3) & \\
%	&=& S(2^{\ell+1}+4r+3) + S(4r+3) 
%	& \text{(by \eqref{eq:rec})}	\\
%	&=& 4 S(2^{\ell-1}+r+1) - S(2^{\ell} + 2r+1) + 4 S(r+1) - S(2r+1) 
%	&\text{(by induction hypothesis)}\\
%	&=& 4 S(n+1) - S(2n+1) 
%	&\text{(by \eqref{eq:rec})}.
%\end{array}
%\]
%\item
%If $2^{\ell-1} \leq r < 2^\ell$, we get
%\[
%\begin{array}{rclr}
%	S(4n+3) 
%	&=& S(2^{\ell+2}+4r+3) & \\
%	&=& S(2^{\ell+3}-4(r+1)+2) 
%	&\text{(by \eqref{eq:rec})} \\
%	&=& 4 S(2^{\ell+1}-(r+1)+1) - S(2^{\ell+2} - 2r-1+1)
%	&\text{(by induction hypothesis)}\\
%	&=& 4 S(n+1) - S(2n+1) 
%	&\text{(by \eqref{eq:rec})}.
%\end{array}
%\]
%\end{enumerate}
\end{enumerate}
To finish the proof, observe that the $\mathbb{Z}$-module $\left\langle \mathcal{K}_2(S)\right\rangle$ is finitely-generated: a choice of generators is $(S(n))_{n\ge 0}$ and $(S(2n))_{n\ge 0}$.  
\end{proof}

If a sequence is $k$-regular, then its $n$th term can be obtained by multiplying some matrices and the length of this product is proportional to $\log_k(n)$. In our situation, we will consider products of square matrices of size $2$, \cite{AS99}, \cite[Theorem 16.1.3]{AS03}.
Observe that due to Equation~\eqref{eq:SB}, other matrices can be derived from the 2-regularity of $(SB(n))_{n \geq 0}$.

\begin{corollary}\label{cor:mat_2}
For all $n\ge 0$, let 
\begin{equation}
 \label{eq:defV(n)}
V(n) = \left(
\begin{array}{c}
 S(n) \\
 S(2n) 
\end{array}
\right).
\end{equation}
Consider the matrix-valued morphism $\mu : \{0,1\}^* \to \mathbb{Z}^2_2$ defined by
$$\mu(0)=\left(\begin{array}{cc}
 0 & 1 \\
 -1 & 2\\
\end{array}\right),\quad
\mu(1)=\left(
\begin{array}{cc}
 3 & -1 \\
 4 & -1 \\
\end{array}
\right).$$
Then $V(2n+r)=\mu(r)V(n)$ for all $r\in\{0,1\}$ and $n\ge 0$. Consequently, if $\rep_2(n)=c_r\cdots c_0$, then 
$$
S(n) =
\begin{pmatrix}
    1&0\\
\end{pmatrix}
\, \mu(c_0)\cdots \mu(c_r)\, V(0).
$$
\end{corollary}

\begin{proof}
Thanks to Theorem~\ref{the:rel}, we directly have
$$
V(2n) = 
\left(
\begin{array}{c}
 S(2n) \\
 S(4n) \\
 \end{array}
\right)
=\left(\begin{array}{cc}
 0 & 1 \\
 -1 & 2\\
\end{array}\right)
\left(
\begin{array}{c}
 S(n) \\
 S(2n) \\
\end{array}
\right)$$
and
$$V(2n+1)= 
\left(
\begin{array}{c}
 S(2n+1) \\
 S(4n+2) \\
\end{array}
\right)
= \left(
\begin{array}{cc}
 3 & -1 \\
 4 & -1 \\
\end{array}
\right) \left(
\begin{array}{c}
 S(n) \\
 S(2n) \\
\end{array}
\right)$$
for all $n\ge 0$. Now let $r=\sum_{i=0}^{\ell-1} r_i\, 2^i$. 
Then the word $r_{\ell-1}\cdots r_0$ is the expression of $r$ in base $2$ possibly with leading zeroes. 
By induction, we can show that
\begin{equation}\label{eq:mu}
V(2^\ell m + r) = \mu(r_0 \cdots r_{\ell-1}) V(m)
\end{equation}
for all $m\in\mathbb{N}$. 
Now let $n\ge 2$. 
Then there exist $\ell\ge 1$ and $r\in\{0, \ldots, 2^\ell -1\}$ such that $n=2^\ell + r$. Let $r_{\ell-1}\cdots r_0$ be the expression of $r$ in base $2$ possibly with leading zeroes.  
Using \eqref{eq:mu} and the fact that $V(1)=\mu(1)V(0)$, we get
$$
V(n) = \mu(r_0 \cdots r_{\ell-1}) V(1) = \mu(r_0 \cdots r_{\ell-1}1) V(0) = \mu((\text{rep}_2(n))^R) V(0)
$$
where $u^R$ is the reversal of the word $u$. 
Note that the previous equality also holds for $n\in\{0,1\}$. In particular, we have the following equality
$$
S(n) = \begin{pmatrix}
    1&0\\
\end{pmatrix}
\, \mu((\text{rep}_2(n))^R V(0)
$$
for all $n\in\mathbb{N}$.
\end{proof}

%========================
\section{The sequence $(S(n))_{n\ge 0}$ is not 2-synchronized}\label{sec:sync}
%===============================================

The class of $k$-synchronized sequences is an intermediate between the classes of $k$-automatic sequences and $k$-regular sequences. Every $k$-synchronized sequence is $k$-regular but the converse does not hold. Moreover, every $k$-synchronized sequence taking finitely many values is $k$-automatic. Roughly, a sequence $(s(n))_{n \ge 0}$ is $k$-synchronized if there exists a finite automaton accepting the pairs of base-$k$ expansions of $n$ and $s(n)$, as stated in Definition~\ref{def:k-sync}. These sequences were first introduced in \cite{CM}. As an example, it is proved in \cite{goc} that if an infinite word $\mathbf{w}$ is $k$-automatic, then its factor complexity function is $k$-synchronized. 

\begin{definition}\label{def:k-sync}
Let $k \geq 2$ be an integer. The map $\rep_k$ is extended to $\mathbb{N}\times\mathbb{N}$ as follows. For all $m,n\in\mathbb{N}$, 
$$\rep_k(m,n)=\left(0^{M-|\rep_k(m)|}\rep_k(m),0^{M-|\rep_k(n)|}\rep_k(n)\right)$$
where $M=\max\{|\rep_k(m)|,|\rep_k(n)|\}$. The idea is that the shortest word is padded with leading zeroes to get two words of the same length.

A sequence $(s(n))_{n \ge 0}$ of integers is said to be {\em $k$-synchronized} if the language $\{\rep_k(n,s(n)) \mid n \in \mathbb{N}\}$ is accepted by some finite automaton reading pairs of digits.
\end{definition}

We present two independent proofs of the next result.

\begin{prop}
The sequence $(S(n))_{n \ge 0}$ is not $2$-synchronized.
\end{prop}
\begin{proof}
Proceed by contradiction and suppose that there is a deterministic $k$-state automaton that accepts exactly the language $\{\rep_2(n,S(n)) \mid n \in \mathbb{N}\}$.
Note that for all $\ell \geq 0$, $S(2^\ell)=\ell+2$ using~\eqref{eq:defS}.
Consider an integer $\ell$ such that $\ell - \lceil \log_2(\ell+2)\rceil > k+1$.
Then we have $\rep_2(2^\ell,\ell+2) = (10^\ell,0^{k+2}u)$ for some word $u \in \{0,1\}^*$ of length $\ell - k -1$ and such that $\val_2(u)=\ell+2$.
Let $(q_0,q_1,\dots,q_{\ell+2})$ be the path in the automaton starting in the initial state $q_0$ and whose label is $\rep_2(2^\ell,\ell+2)$.
Since $\rep_2(2^\ell,\ell+2)$ is accepted by the automaton, the state $q_{\ell+2}$ is an accepting state.
As the automaton has $k$ states, there exist $1 \leq i < j \leq k+2$ such that $q_i = q_j$.
By choice of $\ell$, there is a path from $q_i$ to $q_j$ whose label is $(0^{j-i},0^{j-i})$.
Thus the pair of words $(10^{\ell+j-i},0^{k+2+j-i}u)$ is accepted by the automaton.
However, we have $S(\val_2(10^{\ell+j-i})) = \ell+j-i+2 \neq \val_2(0^{k+2+j-i}u) = \ell+2$, which is a contradiction.
\end{proof}

The fact that the sequence $(S(n))_{n\ge 0}$ is not $2$-synchronized also follows from the next result.

\begin{lemma}\cite[Lemma 4]{SchaSha}\label{lem:SchaSha}
If $(f(n))_{n\ge 0}$ is a $k$-synchronized sequence, and $f \neq O(1)$, then there exists a constant $c > 0$ such that $f(n) \ge cn$ infinitely often.
\end{lemma}
 
We can make use of the growth order of the Stern--Brocot sequence. One can numerically estimate that the joint spectral radius $\rho$ of the matrices $\mu(0)$ and $\mu(1)$ is between $1.61$ and $1.71$. Since, for all $n\ge 0$, $S(n)$ can be computed by multiplying those matrices and the number of multiplications is proportional to $\log_2(n)$, it follows that the growth rate of $S(n)$ cannot be higher than $\rho^{\log_2(n)}=n^{\log_2(\rho)}$ multiplied by a constant. By Lemma~\ref{lem:SchaSha}, the sequence $(S(n))_{n\ge 0}$ cannot be $2$-synchronized. 

%===============================================
\section{Extension to a larger alphabet}\label{sec:4}
%===============================================

In this section, we investigate the case of the base-$k$ numeration system for $k \geq 2$.
We are able to extend our main tool that are the tries of subwords defined in Section~\ref{sec:3}. 
However, it is not straightforward to deduce an analogue of Proposition~\ref{pro:rec} from Corollary~\ref{cor:arbre base k}. Nevertheless, numerical computations suggest that the analogue of the sequence $(S(n))_{n\ge 0}$ in the base-$k$ case seems to be $k$-regular. 
In \cite{LRS3}, we prove an exact behavior for the summatory function of $(S(n))_{n\ge 0}$ and conjecture a similar behavior for larger bases.
General asymptotic estimates for summatory function of $k$-regular sequences are provided by Dumas in \cite{Dumas}.

For instance, if $(S_3(n))_{n\ge 0}$ denotes the analogue of the sequence $(S(n))_{n\ge 0}$ in the base-$3$ case, then the first few terms of this sequence are
$$1, 2, 2, 3, 3, 4, 3, 4, 3, 4, 5, 6, 5, 4, 6, 7, 7, 6, 4, 6, 5, 7, 
6, 7, 5, 6, 4, 5, 7, 8, 8, 7, 10,\ldots.
$$
We conjecture the following two results that are the analogues of Proposition~\ref{pro:rec} and Theorem~\ref{the:rel}. For all $\ell\ge 1$ and $0\le r < 3^\ell$, the sequence $(S_3(n))_{n\ge 0}$ satisfies $S_3(0)=1$, $S_3(1)=S_3(2)=2$,
$$
S_3(3^{\ell}+r)=\left\{
    \begin{array}{ll}
        S_3(3^{\ell-1}+r)+S_3(r),& \text{ if } 0 \le r < 3^{\ell-1};\\
        2 S_3(r)-S_3(r-3^{\ell-1}),& \text{ if } 3^{\ell-1} \le r < 2\cdot 3^{\ell-1};\\
        2 S_3(r)+S_3(r-3^{\ell-1})-2S_3(r-2\cdot 3^{\ell-1})       ,& \text{ if } 2\cdot 3^{\ell-1} \le r < 3^{\ell}; \\
    \end{array}\right.
$$
and
$$
S_3(2\cdot 3^{\ell}+r)=\left\{
    \begin{array}{ll}
        S_3(2\cdot 3^{\ell-1}+r)+S_3(r),& \text{ if } 0 \le r < 3^{\ell-1};\\
        S_3(r+3^{\ell-1})+2S_3(r)-2S_3(r-3^{\ell-1}),& \text{ if } 3^{\ell-1} \le r < 2\cdot 3^{\ell-1};\\
        2 S_3(r)-S_3(r-2\cdot 3^{\ell-1}),& \text{ if } 2\cdot 3^{\ell-1} \le r < 3^{\ell}.\\
    \end{array}\right.
$$
Observe that we divide the statement into two cases depending on the first letter of the base-$3$ expansion of the integers. Moreover, there exists a partial palindromic structure inside of the sequence $(S_3(n))_{n\ge 0}$ that can be stated as follows: for all $\ell\ge 1$ and $0\le r < 3^\ell$,
$$
S_3(2\cdot 3^{\ell}+r) = S_3(2\cdot 3^{\ell}+3^{\ell}-r-1).  
$$

\begin{conj}\label{the:rel3}
The sequence $(S_3(n))_{n \ge 0}$ satisfies, for all $n \geq 0$,
\begin{eqnarray*}
	S_3(3n+2) 	&=& 5\, S_3(n)-S_3(3n)-S_3(3n+1)			\\
	S_3(9n) 	&=& -S_3(n)+2\, S_3(3n)			\\
	S_3(9n+1) 	&=& -2\, S_3(n)+2\, S_3(3n)+S_3(3n+1)			\\
	S_3(9n+3) 	&=& -2\, S_3(n)+S_3(3n)+2\, S_3(3n+1)			\\
	S_3(9n+4) 	&=& -S_3(n)+2\, S_3(3n+1)			\\
	S_3(9n+6) 	&=& 8\, S_3(n)-S_3(3n)-2\, S_3(3n+1)			\\
	S_3(9n+7) 	&=& 8\, S_3(n)-2\,S_3(3n)-S_3(3n+1).
\end{eqnarray*}
In particular, $(S_3(n))_{n\ge 0}$ is $3$-regular.
\end{conj}

For each $k\ge 2$, we let $L_k$ denote the language $\{1,\ldots,k-1\}\{0,\ldots,k-1\}^*\cup\{\varepsilon\}$. Let us introduce some notation.
As in Section~\ref{sec:3}, for each non-empty word $w \in L_k$, we consider an equivalent factorization of $w$ into maximal blocks of letters of the form
\[
	w = a_1^{n_1} \cdots a_M^{n_M},
\]
with $n_\ell \geq 1$ for all $\ell$.
For each $\ell \in \{0,\dots,M-1\}$, we consider the subtree $T_\ell$ of $\mathcal{T}_{L_k}(w)$ whose root is the node $a_1^{n_1} \cdots a_\ell^{n_\ell} a_{\ell+1}$.
We again set $T_M$ to be an empty tree with no node.

For each $\ell \in \{0,\dots,M-1\}$, let $\mathrm{Alph}(\ell)$ denote the set of letters occurring in $a_{\ell+1}\cdots a_M$.
Then for each letter $b \in \mathrm{Alph}(\ell)$, we let $j(b,\ell)$ denote the first index greater than $\ell$ such that $a_{j(b,\ell)} = b$.
We obtain the analogue of Proposition~\ref{pro:forme_arbre}.

\begin{prop}
Let $w$ be a finite word in $L_k$. 
With the above notation about $M$ and the subtrees $T_\ell$, the tree $\mathcal{T}_{L_k}(w)$ has the following properties.
\begin{enumerate}
\item
The node of label $\varepsilon$ has $\#(\mathrm{Alph}(0)\setminus\{0\})$ children that are $b$ for $b \in \mathrm{Alph}(0)\setminus\{0\}$. Each child $b$ is the root of a tree isomorphic $T_{j(b,0)-1}$.
\item
For each $\ell \in \{0,\dots,M-1\}$ and each $i \in \{0,\dots,n_{\ell+1}-1\}$ with $(\ell,i) \neq (0,0)$, the node of label $x = a_1^{n_1} \cdots a_{\ell}^{n_{\ell}}a_{\ell+1}^i$ has $\#(\mathrm{Alph}(\ell))$ children that are $xb$ for $b \in \mathrm{Alph}(\ell)$. Each child $xb$ with $b \neq a_{\ell+1}$ is the root of a tree isomorphic to $T_{j(b,\ell)-1}$.
\end{enumerate}
\end{prop}

\begin{example}
Let $w = 22000112 \in L_3$.
The tree $\mathcal{T}_{L_3}(w)$ is depicted in Figure~\ref{fig:trie5}. We follow the same lines as for Example \ref{exa:tr1}. We use three different colors to represent the letters $0,1,2$.
\begin{figure}[h!tb]
    \centering
    {\psfrag{T0}{$T_0$}\psfrag{T1}{$T_1$}\psfrag{T2}{$T_2$}\psfrag{T3}{$T_3$}
\psfrag{0}{$0$}\psfrag{1}{$1$}\includegraphics{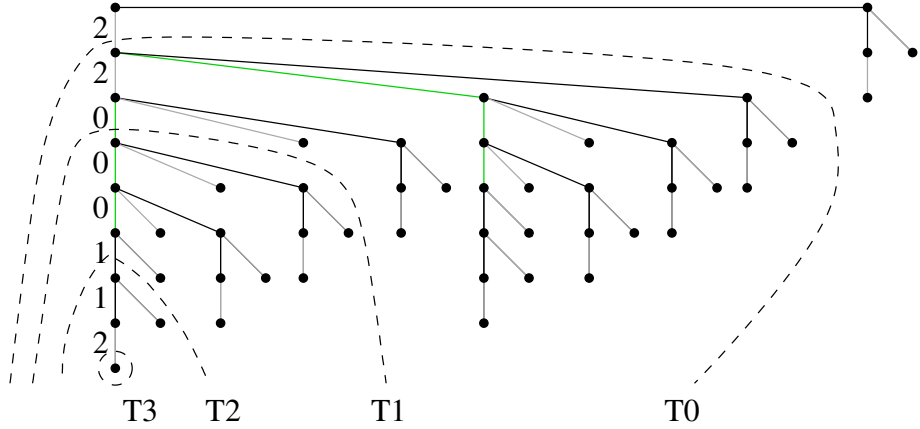}}
    \caption{The trie $\mathcal{T}_{L_3}(22000112)$.}
    \label{fig:trie5}
\end{figure}
\end{example}

With a similar proof, we then get the following analogue of Corollary~\ref{cor:arbreTl}.

\begin{corollary}\label{cor:arbre base k}
Let $w$ be a finite word in $L_k$. 
With the above notation about $M$ and the subtrees $T_\ell$, the following holds.
The number of nodes in $T_{M-1}$ is $n_M$.
For $i=M-2,M-3,\ldots,0$, the number of nodes in $T_i$ is given by
$$
	\# T_i
	= n_{i+1} 
	\left (1+\sum_{\substack{b \in \mathrm{Alph}(i+1) \\ b \neq a_{i+1}}} \# T_{j(b,i+1)-1} \right)
	+ \# T_{j(a_{i+1},i+1)-1}.
$$
The number of subwords in $L_k$ of $w$ is given by $1+\# T_0 +\sum_{\substack{b \in \mathrm{Alph}(1) \\ b \neq 0, a_{1}}} \# T_{j(b,1)-1}$.
\end{corollary}

%===============================================
\section{Generalizations: the Fibonacci case}\label{sec:fibo}
%===============================================

Other generalizations of the integer base numeration system are the linear numeration systems. 
In the previous sections, we considered the language $\{\varepsilon\}\cup 1\{0,1\}^*$. This means that there was no restriction on the words over $\{0,1\}$ except for the presence of a convenient leading $1$. One could naturally define other Pascal triangles by considering any infinite language $K$ over a totally ordered alphabet. If we order the words of $K$ by genealogical ordering: $w_0,w_1,w_2,\ldots$, we obtain an infinite array where the element in the $i$th row and $j$th column, is $\binom{w_i}{w_j}$. This corresponds exactly to the case where we replace base-$2$ expansions with representations within an abstract numeration system based on $K$ \cite{LR}, \cite[Chap.~3]{CANT}. 

We will handle the case of the Fibonacci numeration system, i.e., with the language $L_F=\{\varepsilon\}\cup 1\{0,01\}^*$. 
It turns out that the sequence $(S_F(n))_{n\ge 0}$ counting the number of words in $L_F$ occurring as subwords of the $n$th word in $L_F$ has properties similar to those of $(S(n))_{n\ge 0}$. 
In particular, the notion of $k$-regularity may be extended to take into account a larger class of numeration systems \cite{Sha88,AST}.

As observed in Remark~\ref{rem:tribo foire}, what seems to be important for further generalizations is that the language $K$ of the numeration is the set of words not starting with 0 and not containing occurrences of a set of words of length $2$.
For the Fibonacci numeration system, the language of the numeration is obtained by avoiding the factor $11$.

In the following definition, we recall the notion of a positional numeration system based on a sequence of integers. See, for instance, \cite{Fraenkel,Rigo2} or \cite[Chap.~2]{CANT}.

\begin{definition}
Let $U = (U(n))_{n\ge 0}$ be a sequence of integers such that $U$ is increasing, $U(0)=1$ and $\sup_{n \geq 0} \frac{U(n+1)}{U(n)}$ is bounded by a constant. 
We say that $U$ is a \emph{linear numeration system} if $U$ satisfies a linear recurrence relation, i.e., there exist $a_0, \ldots, a_{k-1} \in \mathbb{Z}$ such that
$$
\forall n \ge 0, \quad U(n+k) = a_{k-1}\, U(n+k-1) + \cdots + a_0\, U(n).
$$
Let $n$ be a positive integer. By successive Euclidean divisions, there exists $\ell\ge 1$ such that 
$$ n = \sum_{j=0}^{\ell -1} c_j\, U(j)
$$
where the $c_j$'s are non-negative integers and $c_{\ell-1}$ is non-zero. The word $c_{\ell -1}\cdots c_0$ is called the \emph{normal $U$-representation} of $n$ and is denoted by $\rep_U(n)$. In other words, the word $c_{\ell-1}\cdots c_0$ is the greedy expansion of $n$ in the considered numeration. We set $\rep_U(0)=\varepsilon$. Finally, we say that $\rep_U(\mathbb{N})$ is the \emph{language of the numeration}. If $d_{\ell -1}\cdots d_0$ is a word over an alphabet of digits, then we set
$$\val_U(d_{\ell -1}\cdots d_0)=\sum_{j=0}^{\ell -1} d_j\, U(j).$$
\end{definition}

We consider the Fibonacci sequence $F=(F(n))_{n\ge 0}$ defined by $F(0)=1$, $F(1)=2$ and $F(n+2)=F(n+1)+F(n)$ for all $n\ge 0$. Using the previous definition, the linear numeration system associated with $F$ is called the \emph{Zeckendorf numeration system} \cite{Zeck} or the \emph{Fibonacci numeration system}. Observe that, for this system, $\rep_F(\mathbb{N}) = L_F$ is given by the set of words over $\{0, 1\}$ avoiding the factor $11$. In Table~\ref{tab:Zeck}, we write the normal $F$-representations of the first few integers. 
\begin{table}[h!t]
$$\begin{array}{c|r||c|r||c|r}
0 & \varepsilon  & 6 & 1001 & 12 & 10101\\
1 & 1 & 7 & 1010 & 13 & 100000 \\
2 & 10 & 8 & 10000 & 14 & 100001  \\
3 & 100 & 9 & 10001 & 15 & 100010 \\
4 & 101 & 10 & 10010 & 16 & 100100 \\
5 & 1000 & 11 & 10100 & 17 & 100101
\end{array}$$ 
\caption{The normal $F$-representations of the first few integers.}
    \label{tab:Zeck}
\end{table}

To define an array, we consider all the words in $L_F$ and we order them using the radix order. In the same way as the generalized Pascal triangle $\mathsf{P}_2$ defined in \cite{LRS}, for a word $u\in L_F$, we compute the binomial coefficient $\binom{u}{v}$ for all words $v\in L_F$. The first few values of the corresponding Pascal triangle with words in $L_F$, denoted by $\mathsf{P}_F$, are given in Table~\ref{tab:coeffFib}. 
\begin{table}[h!t]
$$\begin{array}{r|r|ccccccccc|c}
&&\varepsilon&1&10&100&101&1000&1001&1010&10000&S_F\\
\hline
\rep_F(0)& \varepsilon&1 & 0 & 0 & 0 & 0 & 0 & 0 & 0 & 0 & 1\\
\rep_F(1)& 1&1 & 1 & 0 & 0 & 0 & 0 & 0 & 0 & 0 & 2\\
\rep_F(2)& 10&1 & 1 & 1 & 0 & 0 & 0 & 0 & 0 & 0 & 3\\
\rep_F(3)& 100&1&1 & 2 & 1 & 0 & 0 & 0 & 0 & 0 & 4\\
\rep_F(4)& 101&1 & 2 & 1 & 0 & 1 & 0 & 0 & 0 & 0 & 4\\
\rep_F(5)& 1000&1 & 1 & 3 & 3 & 0 & 1 & 0 & 0 & 0 & 5\\
\rep_F(6)& 1001&1 & 2 & 2 & 1 & 2 & 0 & 1 & 0 & 0 & 6\\
\rep_F(7)& 1010& 1 & 2 & 3 & 1 & 1 & 0 & 0 & 1 & 0 & 6\\
\rep_F(8)&  10000&1 & 1 & 4 & 6 & 0 & 4 & 0 & 0 & 1 & 6\\
\end{array}$$
\caption{The  first few values in the generalized Pascal triangle~$\mathsf{P}_F$ with words in $L_F$.}
    \label{tab:coeffFib}
\end{table}

As before, we are interested in a particular sequence.

\begin{definition}\label{def:S_F}
Let $(S_F(n))_{n\ge 0}$ be the sequence whose $n$th term, $n\ge 0$, is the number of non-zero elements in the $n$th row of $\mathsf{P}_F$. Hence, the first few terms of this sequence are
$$1, 2, 3, 4, 4, 5, 6, 6, 6, 8, 9, 8, 8, 7, 10, 12, 12, 12, 10, 12, 12, 
8, 12, 15, 16, 16, 15, \ldots.$$
Otherwise stated, for $n\ge 0$, 
\begin{equation}
    \label{eq:defSF}
    S_F(n):=\#\left\{v\in L_F\mid \binom{\rep_F(n)}{v}>0\right\}.
\end{equation}
In Figure~\ref{fig:SFn}, we have depicted the sequence $(S_F(n))_{n\ge 0}$ in the interval $[0,F(11)]$.
\begin{figure}[h!tbp]
    \centering
    \includegraphics[scale=1]{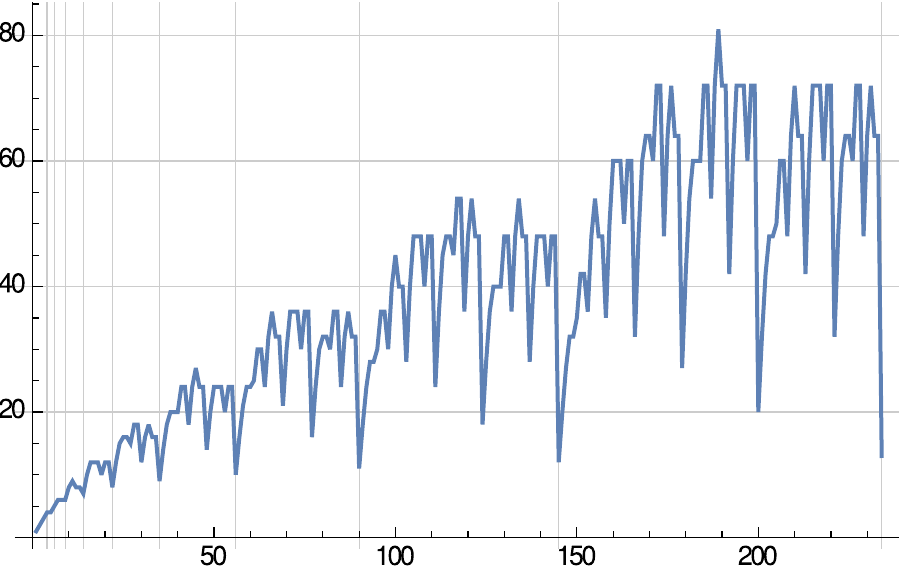}
    \caption{The sequence $(S_F(n))_{n\ge 0}$ in the interval $[0,F(11)]$.}
    \label{fig:SFn}
\end{figure}
\end{definition}

Surprisingly, we will show that the sequence $(S_F(n))_{n\ge 0}$ satisfies a recurrence relation of the same kind as in Proposition~\ref{pro:rec} and is also regular in some extended way \cite{AST,Maes}. 

\begin{prop}\label{pro:recF}
We have $S_F(0)=1$, $S_F(1)=2$. Any integer $n\ge 2$, can be written as $F(\ell)+r$ for some $\ell\ge 1$ and $0 \le r< F(\ell-1)$. We have 
    $$S_F(F(\ell)+r)=\left\{
        \begin{array}{ll}
            S_F(F(\ell-1)+r)+S_F(r),& \text{ if }0\le r<F(\ell-2);\\
            2S_F(r), & \text{ if }F(\ell-2)\le r < F(\ell-1).\\
        \end{array}\right.$$
\end{prop}

The proof is postponed to the end of the section (after Corollary~\ref{cor:matF}). Using this result, there is a convenient way given in Table~\ref{tab:sf} to arrange the terms of the sequence.
\begin{table}[h!t]
$$\begin{array}{ccccccccccccc}
1& \\
2& \\
3& \\
4&4 \\
5&6&6 \\
6&8&9&8&8 \\
7&10& 12& 12& 12& 10& 12& 12 \\
8&12& 15& 16& 16& 15& 18& 18& 12& 16& 18& 16& 16 \\
\end{array}$$
    \caption{Arrangement of the first few terms in $(S_F(n))_{n\ge 0}$.}
    \label{tab:sf}
\end{table}
The $0$th (resp., first) row of this table contains the element $S_F(0)$ (resp., $S_F(1)$) of the sequence. Then, for all $n\ge 2$, the $n$th row contains the elements $S_F(i)$ where $i\in\{F(n-1),\ldots,F(n-1)+F(n-2)-1 \}$ and thus contains $F(n-2)$ elements. Using Proposition~\ref{pro:recF}, for $n\ge 4$, the first $F(n-3)$ elements in the $n$th row are derived from the previous row: the difference of two consecutive rows is a prefix of $(S_F(n))_{n\ge 0}$. The last $F(n-4)$ elements in the $n$th row are twice the elements of the $(n-2)$th row. For $n\ge 1$, also observe that the first element of the $n$th row is equal to $S_F(F(n-1))$.

%=============================================
\section{$F$-regularity of $(S_F(n))_{n\ge 0}$}\label{sec:freg}
%=============================================
We now study the regularity of the sequence $(S_F(n))_{n\ge 0}$. Observe that in Definition~\ref{def:ker}, to obtain one element of the kernel $\mathcal{K}_k(s)$, it is equivalent to consider a word $q\in\{0,\ldots,k-1\}^*$ and all the indices of $s$ whose base-$k$ expansion (with possibly some leading zeroes) ends with the suffix $q$. As an example, $(s(2^3n+1))_{n\ge 0}$ corresponds to the word $q=001$ (in base $2$). 

\begin{definition}
Let $X$ be a subset of $\mathbb{N}$. 
For each $q$ in $\{0,1\}^*$, we define the map 
 $$i_q(X):=\val_F(0^*\rep_F(X)\cap \{0,1\}^*q).$$
In other words, $i_q$ selects elements in $X$ whose normal $F$-representation (padded with leading zeroes) ends with $q$. If $i_q(\mathbb{N})$ is non-empty, it is naturally ordered: $$i_q(\mathbb{N})=\{x_{q,0}<x_{q,1}<x_{q,2}<\cdots\}$$
and by abuse of notation, we set $i_q(n)=x_{q,n}$ for all $0 \leq n < \#i_q(\mathbb{N})$.
\end{definition}

\begin{example} 
The first values in $i_0(\mathbb{N})$ (resp., $i_1(\mathbb{N})$) are $0,2,3,5,7,8$ (resp., $1,4,6,9$). In particular, $i_0(0)=0$, $i_0(1)=2$, \ldots, $i_0(5)=8$. The first values in $i_{10}(\mathbb{N})$ are $2,7$.
\end{example}

The second part of the next lemma is particularly important when considering elements of the $F$-kernel (see Theorem~\ref{the:freg}).

\begin{lemma}\label{lem:pq}
We have $i_{pq}(\mathbb{N})\subseteq i_q(\mathbb{N})$ and 
$$
	i_{pq}(\mathbb{N})=i_{pq}(i_q(\mathbb{N})).
$$
Moreover, if $pq\in L_F$, then
$$
	\rep_F(i_p(\mathbb{N}))q=\rep_F(i_{pq}(\mathbb{N}))
$$
i.e., if $up\in 0^*L_F$ is such that $\val_F(up)$ is the $n$th value $x_{p,n}$ occurring in $i_p(\mathbb{N})$, then $\val_F(upq)$ is the $n$th value $x_{pq,n}$ occurring in $i_{pq}(\mathbb{N})$.
\end{lemma}

Let us illustrate the second part of Lemma~\ref{lem:pq} in Table~\ref{tab:illu}. We can directly determine the $n$th word ending with $01$ in the language of the numeration, from the $n$th word ending with $0$ by simply adding a suffix $1$.
\begin{table}[h!t]
    $$\begin{array}{c|r|r|c}
x_{0,0}&\varepsilon&1&x_{01,0}\\
x_{0,1}&10&101&x_{01,1}\\
x_{0,2}&100&1001&x_{01,2}\\
x_{0,3}&1000&10001&x_{01,3}\\
x_{0,4}&1010&10101&x_{01,4}\\
x_{0,5}&10000&100001&x_{01,5}\\
\end{array}$$
    \caption{Illustration of Lemma~\ref{lem:pq}.}\label{tab:illu}
\end{table}

\begin{remark}\label{rem:tribo foire}
In formal language theory, recall that $Lw^{-1}$ denotes the set of words $v$ such that $vw\in L$. 
The second part of Lemma~\ref{lem:pq} holds because the words in $L_F$ are defined by avoiding the factor $11$.
Indeed, since $11$ has length 2, we have $0^* L_F p^{-1} = 0^*L_F(pq)^{-1}$ when $pq \in L_F$ and $p \neq \varepsilon$. 

This does not always hold, notably when there are longer forbidden factors in the language of the numeration.
Even for the Tribonacci language $L_T$ avoiding the factor $111$, we have $1 \in 0^* L_T 1^{-1}$ and $1 \notin 0^* L_T (11)^{-1}$.
\end{remark}

Definition~\ref{def:k-reg} is replaced for the Fibonacci numeration system by the following notion where subsequences of a given sequence are selected by suffix of normal $F$-representations. This extension was first introduced in \cite{Sha88,AST}.

\begin{definition}
Let $q$ be a word in $\{0,1\}^*$ such that $i_q(\mathbb{N})\neq \emptyset$ and let $s=(s(n))_{n\ge 0}$ be a sequence. The subsequence of $s$ defined by $n \mapsto s(i_q(n))$ is called the \emph{subsequence} of $s$ \emph{with least significant digits equal to $q$}. The set of all these subsequences when $q$ belongs to $\{0,1\}^*$ and is such that  $i_q(\mathbb{N})\neq \emptyset$ is called the \emph{Fibonacci-kernel} or \emph{$F$-kernel} of the sequence $s$ and is denoted by $\mathcal{K}_F(s)$. We say that $s$ is \emph{Fibonacci-regular} or \emph{$F$-regular} if $\langle\mathcal{K}_F(s)\rangle$ is a finitely-generated $\mathbb{Z}$-module.
\end{definition}

Each non-empty word in $L_F$ is factorized into consecutive maximal blocks of zeroes separated by a unique one. It is of the form 
$$10^{n_k}10^{n_{k-1}}\cdots 10^{n_2} 10^{n_1}$$
with $n_1\ge 0$ and $n_2,\ldots,n_k>0$. Depending on the form of $u\in L_F$, we have a formula to count the number of words $v\in L_F$ such that $\binom{u}{v}>0$.

\begin{prop}\label{prop:formules-Fib} 
Let $u$ be a non-empty word in $L_F$ of the form $10^{n_k}10^{n_{k-1}}\cdots 10^{n_2} 10^{n_1}$ 
with $n_1\ge 0$ and $n_2,\ldots,n_k>0$. Then
$$
\#\left\{v\in L_F\mid \binom{u}{v}>0\right\} = (n_1+2) \cdot \prod_{j=2}^k (n_j + 1).
$$
\end{prop}

To prove this result, we essentially follow the same ideas as for Corollary~\ref{cor:arbreTl} but we need to consider the tree of subwords restricted to $L_F$. 

For a word $w\in L_F$, the tree $\mathcal{T}_{L_F}(w)$ is given in Definition~\ref{def:trie}. The factorization \eqref{eq:factoriz} of words in $L_F$ has a very particular form (because there is no $11$). To refer to the same subtrees as in Definition~\ref{def:subtree}, we stick to the notation of \eqref{eq:factoriz} even though the blocks of ones are limited to a single digit
$$w=\underbrace{1}_{u_1} \underbrace{0^{n_2}}_{u_2} \underbrace{1}_{u_3} \underbrace{0^{n_4}}_{u_4} \cdots \underbrace{1}_{u_{2j-1}} \underbrace{0^{n_{2j}}}_{u_{2j}}$$
with $j\ge 1$, $n_2,\ldots, n_{2j-2}\ge 1$ and $n_{2j}\ge 0$.
Let $M=M_w$ be such that $w=u_1u_2\cdots u_M$ where $u_M$ is the last non-empty block of zeroes or the last one.

\begin{example}
Consider the word $w=101000100$. With the above notation, $M=6$. In Figure~\ref{fig:treeF1}, we have represented the trie of subwords $\mathcal{T}_{L_F}(w)$ in $L_F$ and the subtrees $T_0,\ldots,T_5$. The roots of these subtrees correspond to a prefix of $w$ ending with $1$ or $10$.
\begin{figure}[h!tb]
    \centering
    {\psfrag{T0}{$T_0$}\psfrag{T1}{$T_1$}\psfrag{T2}{$T_2$}\psfrag{T3}{$T_3$}\psfrag{T4}{$T_4$}\psfrag{T5}{$T_5$}
\psfrag{0}{$0$}\psfrag{1}{$1$}
\includegraphics{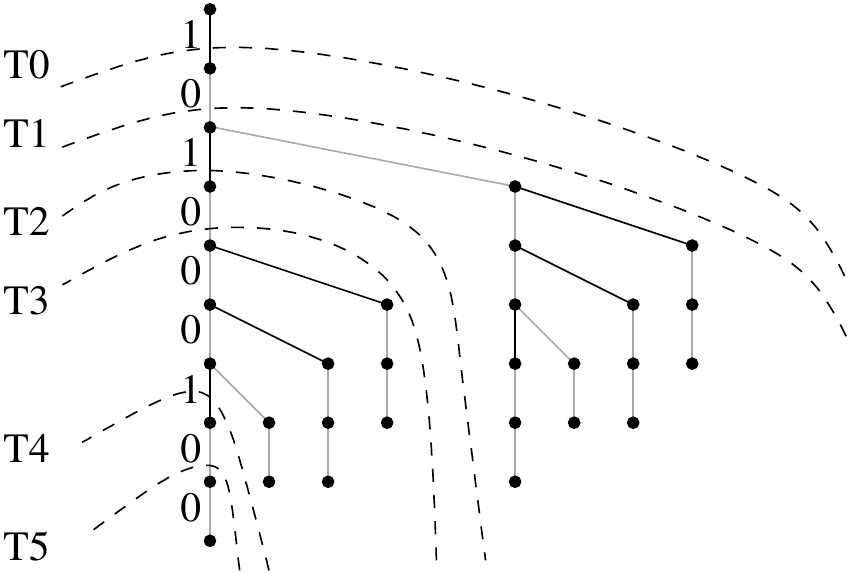}}
    \caption{The trie $\mathcal{T}_{L_F}(101000100)$.}
    \label{fig:treeF1}
\end{figure}
Figure~\ref{fig:treeF2} illustrates the next proposition. We see how the subtrees are connected to the ``initial'' linear subtree labeled by $w$.
\begin{figure}[h!tb]
    \centering
    {\psfrag{T0}{$T_0$}\psfrag{T1}{$T_1$}\psfrag{T2}{$T_2$}\psfrag{T3}{$T_3$}\psfrag{T4}{$T_4$}\psfrag{T5}{$T_5$}
\psfrag{0}{$0$}\psfrag{1}{$1$}
\includegraphics{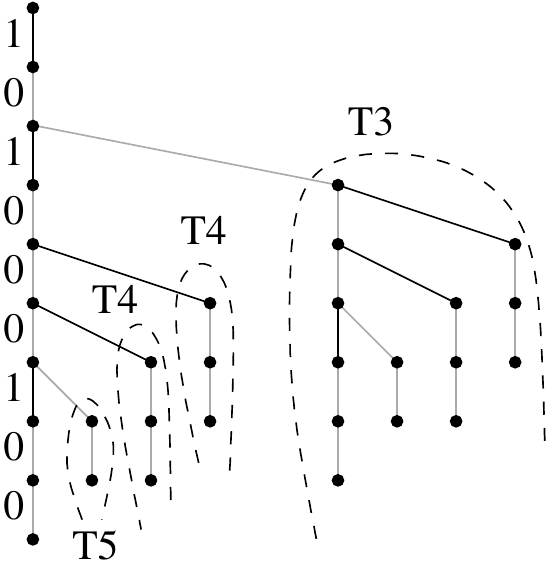}}
    \caption{The trie $\mathcal{T}_{L_F}(101000100)$.}
    \label{fig:treeF2}
\end{figure}
\end{example}

Since we are considering the language $L_F$, the analogue of Proposition~\ref{pro:forme_arbre} becomes:
\begin{prop}\label{pro:forme_arbre2}
Let $w$ be a finite word in $L_F$. With the above notation about $M$ and the subtrees $T_\ell$,
the tree $\mathcal{T}_{L_F}(w)$ has the following properties.
\begin{itemize}
\item
Assume that $2\le 2k<M$.
For every $j\in\{1,\ldots,n_{2k}-1\}$, the node of label $x=u_1 \cdots u_{2k-1} 0^j$ has two children $x0$ and $x1$.
The node $x1$ is the root of a tree isomorphic to $T_{2k}$. 
Moreover, $x=u_1 \cdots u_{2k-1}$ has a single child $x0$.

\item
Assume that $3\le 2k+1<M$.
The node of label $x=u_1 \cdots u_{2k}$ has two children $x0$ and $x1$.
The node $x0$ is the root of a tree isomorphic to $T_{2k+1}$. 

\item
The node of label $x=\varepsilon$ has only one child $x1$. 
\end{itemize}
\end{prop}

\begin{proof}[Proof of Proposition~\ref{prop:formules-Fib}.]
    Proceed by induction and consider two cases depending on the last letter of $u$. Start with a linear tree and add, with a bottom-up approach, all the possible subtrees given by Proposition~\ref{prop:formules-Fib}.
\end{proof}

\begin{theorem}\label{the:freg}
The sequence $(S_F(n))_{n\ge 0}$ satisfies, for all $n\ge 0$, 
\begin{align*}
S_F(i_{00}(n)) &= 2 S_F(i_{0}(n)) - S_F(i_{\varepsilon}(n)) \\
S_F(i_{01}(n)) &= 2 S_F(i_{\varepsilon}(n)) \\
S_F(i_{010}(n)) &= 3S_F(i_\varepsilon (n)). 
\end{align*}
In particular, $(S_F(n))_{n\ge 0}$ is $F$-regular.
\end{theorem}
\begin{proof}
Let $q\in\{0,1\}^*$ such that $i_q(\mathbb{N})\neq \emptyset$. Recall that, for all $n\ge 0$, 
$$
S_F(i_{q}(n)) =\#\left\{v\in L_F\mid \binom{\rep_F(i_q(n))}{v}>0\right\}.
$$

Let us prove the first relation. Let $u$ be a non-empty word in $L_F$. The idea is to use Proposition~\ref{prop:formules-Fib}. 
%Thus, we split the proof into two cases according to the form of $u$. If $u$ is of the form $10^{n_k}10^{n_{k-1}}\cdots 10^{n_2}1$ with  $n_2,\ldots,n_k>0$, then Proposition~\ref{prop:formules-Fib} tells us that
%$$ \#\left\{v\in L_F\mid \binom{u}{v}>0\right\} =2 \cdot \prod_{j=2}^{k} (n_{j} + 1).$$
%The word $u0$ (resp., $u00$) is of the form $10^{n_{k}}10^{n_{k-1}}\cdots 10^{n_2}10^{n_1}$ with  $n_1=1$ (resp., $10^{n_{k}}10^{n_{k-1}}\cdots 10^{n_2}10^{n_1'}$ with $n_1'=2$). Thanks to Proposition~\ref{prop:formules-Fib}, we have
%$$\#\left\{v\in L_F\mid \binom{u0}{v}>0\right\} = (n_1+2)\cdot \prod_{j=2}^{k} (n_j + 1)= 3\cdot \prod_{j=2}^{k} (n_{j} + 1)$$
%and
%$$\#\left\{v\in L_F\mid \binom{u00}{v}>0\right\} = (n_1'+2)\cdot \prod_{j=2}^{k} (n_j + 1)= 4\cdot \prod_{j=2}^{k} (n_{j} + 1).$$
%Finally, we have
%\begin{equation}
%    \label{eq:u00}
%2\cdot \#\left\{v\in L_F\mid \binom{u0}{v}>0\right\} - \#\left\{v\in L_F\mid \binom{u}{v}>0\right\}= \#\left\{v\in L_F\mid \binom{u00}{v}>0\right\}.
%\end{equation}
Let us write $u$ as $10^{n_k}10^{n_{k-1}}\cdots 10^{n_1}$ with  $n_2,\ldots,n_k>0$ and $n_1 \geq 0$.
By Proposition~\ref{prop:formules-Fib} we have 
\begin{eqnarray*} 
	\#\left\{v\in L_F\mid \binom{u}{v}>0\right\} 
	&=& 
	(n_1+2) \cdot \prod_{j=2}^{k} (n_{j} + 1)		\\
	\#\left\{v\in L_F\mid \binom{u0}{v}>0\right\}
	&=& 
	(n_1+3)\cdot \prod_{j=2}^{k} (n_{j} + 1)			\\
	\#\left\{v\in L_F\mid \binom{u00}{v}>0\right\}
	&=& 
	(n_1+4)\cdot \prod_{j=2}^{k} (n_{j} + 1).
\end{eqnarray*}
%The word $u0$ (resp., $u00$) is of the form $10^{n_{k}}10^{n_{k-1}}\cdots 10^{n_2} 10^{n'_1}$ with  $n'_1=n_1+1$ (resp., $n'_1=n_1+2$). Thanks to Proposition~\ref{prop:formules-Fib}, we have
%$$
%\#\left\{v\in L_F\mid \binom{u0}{v}>0\right\}= (n_1+3)\cdot \prod_{j=1}^{k} (n_{j} + 1)$$
%and
%$$
%\#\left\{v\in L_F\mid \binom{u00}{v}>0\right\}= (n_1+4)\cdot \prod_{j=1}^{k} (n_{j} + 1).$$
Hence, we get
\[
	2\cdot \#\left\{v\in L_F\mid \binom{u0}{v}>0\right\} 
	- \#\left\{v\in L_F\mid \binom{u}{v}>0\right\}
	=
	\#\left\{v\in L_F\mid \binom{u00}{v}>0\right\}.
\]
This leads to the expected relation because
$$
	\val_F(\rep_F(n)0)=i_0(n) \text{ and }\val_F(\rep_F(n)00)=i_{00}(n).
$$
These latter relations are reflected by the second part of Lemma~\ref{lem:pq}.

The last two relations are obtained using the same ideas. One has to simply use Proposition~\ref{prop:formules-Fib} with words of the form $u$, $u01$ and $u010$. We derive that 
$$\#\left\{v\in L_F\mid \binom{u01}{v}>0\right\}= 2 \cdot \#\left\{v\in L_F\mid \binom{u}{v}>0\right\}$$
and
$$\#\left\{v\in L_F\mid \binom{u010}{v}>0\right\}= 3 \cdot \#\left\{v\in L_F\mid \binom{u}{v}>0\right\}.$$
The $F$-regularity of the sequence follows from Lemma~\ref{lem:pq}: the $\mathbb{Z}$-module generated by the $F$-kernel of $S_F$ is generated by $(S_F(i_\varepsilon(n)))_{n\ge 0}=(S_F(n))_{n\ge 0}$ and $(S_F(i_0(n)))_{n\ge 0}$. As an example, 
$$
	(S_F(i_{1001}(n)))_{n\ge 0}=2(S_F(i_{10}(n)))_{n\ge 0}=2(S_F(i_{010}(n)))_{n\ge 0}=6(S_F(n))_{n\ge 0}.
$$
\end{proof}

Corollary~\ref{cor:mat_2} is replaced by the following result, which says that if a sequence is $F$-regular, then its $n$th term can be obtained by multiplying some matrices. The length of this product is proportional to $\log_\varphi(n)$ where $\varphi$ is the golden ratio. Here again, we get matrices of size $2$ thanks to Theorem~\ref{the:freg}.

\begin{corollary}\label{cor:matF}
Let 
\begin{equation}
 \label{eq:defV_F(n)}
V(0) = \left(
\begin{array}{c}
 S_F(i_{\varepsilon}(0)) \\
 S_F(i_{0}(0)) 
\end{array}
\right)
=\left(
\begin{array}{c}
 1 \\
 1 
\end{array}
\right).
\end{equation}
Consider the matrix-valued morphism $\mu : \{0,01\}^* \to \mathbb{Z}^2_2$ defined by
$$
\mu(0)=
\left(\begin{array}{cc}
 0 & 1 \\
 -1 & 2\\
\end{array}\right),\quad
\mu(01)=\left(
\begin{array}{ccccc}
 2 & 0\\
 3 & 0\\
\end{array}
\right).
$$
For all $n\ge 0$, let $0\rep_F(n)=u_k\cdots u_1$ where $u_i \in \{0, 01\}$ for all $i\in\{1,\ldots,k\}$. Then
$$
S_F(n) =
\begin{pmatrix}
    1&0\\
\end{pmatrix}
\, \mu(u_1)\cdots \mu(u_k)\, V(0).
$$
\end{corollary}

\begin{proof}
We proceed by induction on $k$ in the factorization $0\rep_F(n)=u_k\cdots u_1$ with $u_i \in \{0, 01\}$ for all $i$.
One can observe that the result is true for $k \in \{1,2\}$.

Assume now that $0\rep_F(n)=u_k\cdots u_1$ with $k \geq 3$.
We only consider the case $u_2 = 0$, the other one is similar.
If $u_1 = 0$, then $S_F(n) = S_F(i_{00}(m))$ for some $m \ge 0$.
By Theorem~\ref{the:freg} and using the induction hypothesis, we get
\begin{eqnarray*}
	S_F(n) 	&=&	2 S_F(i_0(m)) - S_F(i_\varepsilon(m))		\\
			&=&	2 S_F(\val_F(u_k \cdots u_2)) - S_F(\val_F(u_k \cdots u_3))	\\
			&=&	\begin{pmatrix}
    					1&0\\
				\end{pmatrix}
				(2 \mu(0)-I)				
				\, \mu(u_3)\cdots \mu(u_k) \, V(0).
\end{eqnarray*}
The equality follows by observing that $(2 \mu(0)-I) = \mu(0)^2$.
If $u_1 = 01$, then $S_F(n) = S_F(i_{01}(m))$ for some $m \ge 0$.
By Theorem~\ref{the:freg} and using the induction hypothesis, we get
\begin{eqnarray*}
	S_F(n) 	&=& 2 S_F(i_\varepsilon(m))		\\
			&=& 2 S_F(\val_F(u_k \cdots u_2))	\\
			&=&	2 
				\begin{pmatrix}
    					1&0\\
				\end{pmatrix}
				\, \mu(u_2)\cdots \mu(u_k) \, V(0).
\end{eqnarray*}
The equality follows by observing that $2 \begin{pmatrix}  1&0 \end{pmatrix} = \begin{pmatrix}  1&0 \end{pmatrix} \mu(01)$.
%Let us proceed by induction on $n\ge 1$. The idea is to cut $\rep_F(n)$ into blocks belonging to $\{0, 01\}$. To do so, we need to add a zero before $\rep_F(n)$ since $\rep_F(n)$ begins with $1$. It means that, if $0\rep_F(n)=u_k\cdots u_1$ where $u_i \in \{0, 01\}$ for all $i\in\{1,\ldots,k\}$, then we have $u_k=01$. Consider the case when $n=1$. We can easily check that
%$$
%S_F(1) =
%\begin{pmatrix}
%    1&0\\
%\end{pmatrix}
%\, \left(
%\begin{array}{ccccc}
% 2 & 0\\
% 3 & 0\\
%\end{array}
%\right) \, V(0).
%$$
%Now assume the result holds for all $m<n$ and let us prove it for $n$. Let us write $0\rep_F(n)=u_k\cdots u_1$ where $u_i \in \{0, 01\}$ for all $i\in\{1,\ldots,k\}$. Using the induction hypothesis for the second equality, we get
%\begin{align*}
%\begin{pmatrix}
%    1&0\\
%\end{pmatrix}\, \mu(u_{k})\cdots \mu(u_1)\, V(0)&=\begin{pmatrix}
%    1&0\\
%\end{pmatrix}
%\, \left(
%\begin{array}{ccccc}
% 2 & 0\\
% 3 & 0\\
%\end{array}
%\right) \, \mu(u_{k-1})\cdots \mu(u_1)\, V(0) \\ 
%&= \begin{pmatrix}
%    1&0\\
%\end{pmatrix}
%\, \left(
%\begin{array}{ccccc}
% 2 & 0\\
% 3 & 0\\
%\end{array}
%\right) \, \begin{pmatrix}
%    S_F(\val_F(u_{k-1}\cdots u _1))\\
%    \star\\
%\end{pmatrix} \\
%&= 2 \, S_F(\val_F(u_{k-1}\cdots u _1)) \\
%&= S_F(\val_F(u_{k}\cdots u _1))
%\end{align*}
%where the last
\end{proof}

\begin{proof}[Proof of Proposition~\ref{pro:recF}.]
We make use of the previous corollary.
Assume that $n = F(\ell)+r$ with $\ell \geq 1$ and $0 \leq r < F(\ell-1)$.
We have $0 \rep_F(n) = u_k \cdots u_1$ for some $k \geq 2$, with $u_i \in \{0,01\}$ for all $i$.

If $F(\ell-2) \leq r < F(\ell-1)$, then $u_{k-1} = 01$.
By Corollary~\ref{cor:matF}, we get
\begin{eqnarray*}
	S_F(F(\ell)+r)
	&=&
	\begin{pmatrix}
   		1&0\\
	\end{pmatrix}
	\, \mu(u_1)\cdots \mu(u_{k-2}) \mu(01) \mu(01) \, V(0);	\\
	S_F(r)
	&=&
	\begin{pmatrix}
   		1&0\\
	\end{pmatrix}
	\, \mu(u_1)\cdots \mu(u_{k-2}) \mu(01) \, V(0)
\end{eqnarray*}
and the equality $S_F(F(\ell)+r) = 2 S_F(r)$ follows.

If $0 \le r < F(\ell-2)$, then $u_{k-1} = 0$.
Let $m < k-1$ be the greatest integer such that $u_m = 01$ (we set $m=0$ if $u_i = 0$ for all $i \leq k-1$).
By Corollary~\ref{cor:matF}, we have 
\begin{eqnarray*}
	S_F(F(\ell)+r)
	&=&
	\begin{pmatrix}
   		1&0\\
	\end{pmatrix}
	\, \mu(u_1)\cdots \mu(u_m) \mu(0)^{k-m-1} \mu(01) \, V(0);	\\
	S_F(F(\ell-1)+r)
	&=&
	\begin{pmatrix}
   		1&0\\
	\end{pmatrix}
	\, \mu(u_1)\cdots \mu(u_m) \mu(0)^{k-m-2} \mu(01) \, V(0)	;\\
	S_F(r)
	&=&
	\begin{pmatrix}
   		1&0\\
	\end{pmatrix}
	\, \mu(u_1)\cdots \mu(u_m) \, V(0).
\end{eqnarray*}
We get the equality $S_F(F(\ell)+r) = S_F(F(\ell-1)+r) +S_F(r)$ by showing by induction that, for all $n$, 
\[
	\mu(0)^n \mu(01) = 	
	\begin{pmatrix}
  		n+2&0 \\
   		n+3&0
	\end{pmatrix}.
\]
\end{proof}

\section{Concluding remarks}

After considering the language $L_2$ of base-$2$ expansions, then the language $L_F$ of Fibonacci expansions, one can naturally wonder whether similar properties can be observed for an arbitrary initial language (because Pascal-like triangles may be defined in this general setting). 
A first generalization of the Fibonacci case would be to consider the $m$-bonacci case where the corresponding language is made of the words over $\{0,1\}$ avoiding the factor $1^m$ (the Fibonacci case is $m=2$ and the Tribonacci case $m=3$ has been considered in Remark~\ref{rem:tribo foire}). 

\begin{example}
    For $m=3$, the sequence counting admissible subwords associated with the Tribonacci numeration system starts with $$1, 2, 3, 3, 4, 5, 5, 5, 7, 8, 6, 7, 7, 6, 9, 11, 9, 11, 12, 10, 9, 
11, 11, 9, 7, 11, 14, 12, 15, $$
$$17, 15, 14, 18, 19, 15, 14, 14, 11, 15, 
17, 15, 15, 17, 15, 8, 13, 17, 15, 19, 22,\ldots .$$
\end{example}

Due to Remark~\ref{rem:tribo foire} that extends easily to $m$-bonacci languages for $m>3$, Lemma~\ref{lem:pq} does not hold for the $m$-bonacci numeration system as soon as $m \geq 3$. Consequently, it is not clear whether the sequence $(S_T(n))_{n\ge 0}$ is $T$-regular or the analogue sequence for the $m$-bonacci case is $m$-bonacci-regular.
Nevertheless, the sequence $S_T$ seems to partially satisfy a relation similar to the first part of Proposition~\ref{pro:recF}. To build a table similar to the arrangement found in Table~\ref{tab:sf}, numerical observations lead to the following conjecture: if $n_i$ denotes the position of the last occurrence of $i$ in $S_T$ (assuming that $n_i$ is thus well defined, which is the case for the Fibonacci language, as shown by Proposition \ref{pro:ni} below), then $S_T(n_i+r)=S_T(n_{i-1}+r)+S_T(r)$ for $0\le r<n_i-n_{i-1}$, $i\ge 5$. 
\begin{table}[h!t]
$$\begin{array}{cccccccccccccccccccccc}
1\\ 
2 & 3\\ 
3\\
4 & 5 & 5\\
5 & 7 & 8 & 6 & 7 & 7\\
6 & 9 & 11 & 9 & 11 & 12 & 10 & 9 & 11 & 11 & 9\\ 
7 & 11 & 14 & 12 & 15 & 17 & 15 & 14 & 18 & 19 & 15 & 14 & 14 & 11 & 15 & 17 & 15 & 15 & 17 & 15\\
8 & 13 & 17 & 15 & 19 & 22 & \ldots\\
\end{array}$$
    \caption{Arrangement of the first few terms in $(S_T(n))_{n\ge 0}$.}
    \label{tab:st}
\end{table}
Moreover, the sequence $$(n_i)_{i\ge 1}= 0, 1, 3, 4, 7, 13, 24, 44, 81, 149, 274, 504,\ldots$$ satisfies the same linear relation as the Tribonacci sequence when $i\ge 4$. Nevertheless, it is not clear that one can determine a ``simple'' relation for $S_T(n_i+r)$ when $n_i-n_{i-1}\le r< n_{i+1}-n_i$ (corresponding to the second part of Proposition~\ref{pro:recF}) and thus derive a possible regularity of the sequence $S_T$. Again, one can also try with larger values of the parameter $m$ and imagine partial relations of the same form. 

In the Fibonacci case, if $n_i$ denotes the position of the last occurrence of $i$ in $(S_F(n))_{n \geq 0}$ for all $i\ge 1$, then $S_F(n_i + r) = S_F(n_{i-1}+r)+S_F(r)$ for $0\le r<n_i-n_{i-1}$ and $i\ge 5$. Indeed, we just need to combine Proposition~\ref{pro:recF} and Proposition~\ref{pro:ni}.

\begin{prop}\label{pro:ni}
For all $i\ge 1$, the number of occurrences of $i$ in $(S_F(n))_{n \geq 0}$ is finite. 
If $n_i$ denotes the position of the last occurrence of $i$ in $(S_F(n))_{n \geq 0}$, then $n_i = F(i-2)$ for all $i\ge 5$.
\end{prop}
\begin{proof}
First of all, Proposition~\ref{pro:recF} implies that 
\begin{equation}\label{eq:S_(n)>1}
S_F(n) > 1 = S_F(0) \quad \forall \, n\ge 1.
\end{equation}
Then for all $\ell \ge 0$, using Proposition~\ref{pro:recF} leads us to 
\begin{equation}\label{eq:S_Frec}
S_F(F(\ell)) = S_F(F(\ell-1)) + 1 
= S_F(F(\ell-2)) + 2
= \cdots 
= S_F(F(0)) + \ell 
= 2 + \ell. 
\end{equation}
Now we prove the following result : for all $\ell\ge 3$, $S_F(F(\ell))$ is less than $S_F(F(\ell)+r)$ for $0<r< F(\ell-1)$. We show this by induction on $\ell$. If $\ell\in\{3,4\}$, then Table~\ref{tab:sf} gives the result. We suppose the result holds true up to $\ell-1$ and we show it also holds true for $\ell$. For $0 < r < F(\ell-2)$, we have
\[
\begin{array}{rclr}
	S_F(F(\ell)+r)
	&=& S_F(F(\ell-1)+r) + S_F(r)  
	& \text{(by Proposition~\ref{pro:recF})} \\
	&>& S_F(F(\ell-1)) + S_F(0)  
	& \text{(by induction hypothesis and by \eqref{eq:S_(n)>1})} \\
	&>& S_F(F(\ell))  
	& \text{(by Proposition~\ref{pro:recF}).} 
\end{array}
\]
Now, for $F(\ell-2) \le r < F(\ell-1)$, we have $S_F(F(\ell)+r)=2S_F(r)$ by Proposition~\ref{pro:recF}. There exists $0 \le r' < F(\ell-3)$ such that $r=F(\ell-2)+r'$. We get
\[
\begin{array}{rclr}
	S_F(F(\ell)+r)
	&\ge& 2 S_F(F(\ell-2))  
	& \text{(by induction hypothesis)} \\
	&\ge & 2\ell  & \text{(by \eqref{eq:S_Frec})} \\
	&>&2+\ell& \text{(since } \ell\ge 5) \\
    &> &S_F(F(\ell))	& \text{(by \eqref{eq:S_Frec}).}
\end{array}
\]
This ends the proof of the intermediate result. Observe that, if $\ell\ge 3$, we have that $S_F(F(\ell)) < S_F(n)$ for all $n> F(\ell)$ since $S_F(F(\ell)) < S_F(F(\ell+m))$ with $m\ge 1$.

For all $i\ge 1$, we conclude that the number of occurrences of $i$ in $(S_F(n))_{n \geq 0}$ is finite. The sequence $(n_i)_{i\ge 1}$ is thus well-defined. We also get that $n_i = F(i-2)$ for all $i\ge 5$.
\end{proof}

\section*{Acknowledgement}
We would like to thank both reviewers for their insightful comments on the paper. One of the reviewer suggested the connection with the Stern--Brocot sequence and pointed out the reference of Lemma~\ref{lem:SchaSha}.

\end{document}